\documentclass{article}
\usepackage{amsmath}
\usepackage{amssymb}
\usepackage{amsthm}
\usepackage{graphicx}
\usepackage{mathrsfs}

\newtheorem{theorem}{Theorem}[section]

\newtheorem{proposition}[theorem]{Proposition}
\newtheorem{definition}[theorem]{Definition}

\newtheorem{remark}[theorem]{Remark}

\newcommand{\EE}{{\mathbb E}}
\newcommand{\ZZ}{{\mathbb Z}}
\newcommand{\NN}{{\mathbb N}}
\newcommand{\PP}{{\mathbb P}}

\newcommand{\RR}{{\mathbb R}}
\newcommand{\Sets}{{\mathbb S}}
\newcommand{\T}{{\cal T}}
\newcommand{\aP}{{\bf P}}
\newcommand{\aE}{{\bf E}}
\newcommand{\A}{{\cal A}}
\newcommand{\B}{{\cal B}}
\newcommand{\Li}{{L}}

\newcommand{\K}{{\cal K}}

\newcommand{\C}{{\cal C}}

\newcommand{\J}{{\cal J}}

\newcommand{\X}{{\cal X}}

\newcommand{\aS}{{\cal S}}
\newcommand{\aV}{{\cal V}}
\newcommand{\wX}{{\widehat{X}}}
\newcommand{\wT}{{\widehat{T}}}
\newcommand{\wpi}{{\widehat{\pi}}}

\newcommand{\ind}{{\bf 1}}

\begin{document}

\title{\itshape A distribution weighting a set of laws whose 
initial states are grouped into classes}
\author{Servet Martinez}

\maketitle

\begin{abstract}

Let $I$ be a finite alphabet and $\aS\subset I$ be a 
nonempty strict subset. The sequences in $I^\ZZ$
are organized into connected regions which always start 
with a symbol in $\aS$. The regions are labelled by types $C(s)$, 
thus a region starting at $s'\in C(s)$ has the same type as 
one starting at $s$. 
Let $(\aP_s: s\in \aS)$ be a family of distributions 
on $I^\NN$ where each $\aP_s$ charges sequences 
starting with the symbol $s$. We can define a
natural distribution $\PP$ on $I^\NN$, that counts the number of
visits to the states from $\aP_s$, properly weighted.
A dynamics of interest is such that at the
first occurrence of $s'\in \aS\setminus C(s)$ 
the law regenerates with distribution $\aP_{s'}$. In this case
we are able to find simple conditions for 
$\PP$ to be stationary. In addition, we study the following more 
complex model: once a symbol $s'\in \aS\setminus C(s)$ has been encountered,
there is a decision to be made, either a new region 
of type $C(s')$ governed by $\aP_{s'}$ starts 
or the region continues to be a $C(s)$ region. This decision 
is modeled as random and depends on $s'$. In this setting 
a similar distribution to $\PP$ can be constructed and the 
conditions for stationarity are supplied.
These models are inspired by genomic sequences where $I$ is the 
set of codons, the classes $(C(s): s\in \aS)$ group 
codons defining similar genomic classes, e.g. in bacteria 
there are two classes corresponding to the start and stop 
codons, and the random decision to continue a region or to begin a new region of a different class reflects the well-known fact that not every appearance of a start codon marks the beginning of a new coding region.


\end{abstract}

\bigskip


\medskip


\section{Introduction}
\label{Sec1}

Here we give an abstract description of the linear organization of sequences
into different types of regions whose beginnings are marked by a distinguished
number of symbols. The regions are organized in a sequential 
way, each one starts at some prescribed set of symbols and ends 
at some other fixed set which is also the initial symbols 
of a region of a different type. In bacterial genomes there are 
two types of regions: genic and intergenic. Start codons mark the site 
where translation into a polypeptide sequence begins and stop codons 
define where the translation ends. So, stop codons define the 
beginnings of intergenic regions. In our model, 
we assume that there could be an arbitrary number of types.

\medskip

Let $I$ be an alphabet of symbols. The infinite sequences 
of symbols $I^\NN$ are assumed to be organized into connected 
regions labelled by different types. Let $\aS$ be the subset of 
symbols marking the beginning of a region and assume 
it is partitioned into equivalence classes $(C(s): s\in \aS)$, 
each class defining a different type of region. In genomics 
the alphabet $I$ is the set of $64$ codons which are triplets of
the bases $\{A,C,G,T\}$. The set $\aS$ is constituted by the codons
$\{ATG, GTG, TTG, TAA, TAG, TGG\}$, the first three 
are the starting codons for genic regions and the other three 
are the stopping codons marking their ends, so 
$\{ATG, GTG, TTG\}$ and $\{TAA, TAG, TGG\}$ are the two classes.

\medskip

In Section \ref{sec1'x} we supply the main results of our work, but many of
the concepts and intermediate results are presented in Sections 
\ref{sec1x}, \ref{sec2x} and \ref{secren1}. In fact, the proofs of 
the main results employ similar computations to those 
used in the simpler models introduced in the initial sections.

\medskip

Our work has as input a class of distributions 
$(\aP_s:s\in \aS)$ on $I^\NN$.
The law $\aP_s$ {\it governs a region} starting with $s$
and it is said to be of {\it type } $C(s)$.
In Theorem \ref{cprobx} of  Section \ref{sec1x}, we show that there 
is a natural distribution $\PP$ on $I^\NN$ that allows the 
distributions $(\aP_s: s\in \aS)$ to be mixed. This probability measure depends 
on a vector of positive weights $\pi=(\pi_s: s\in \aS)$ and it
counts the number of visits to the states as in Kac's construction
of the stationary vector in Markov chains 
(see Chapter I in \cite {as} or Chapter $10$ in \cite{mt}).

\medskip

In Section \ref{sec2x} we assume the laws have a regenerative structure.
If we start at  $s\in\aS$, the sequence of letters
evolves with the distribution $\aP_{s}$ until $T^1$
which is the time (or site) when a 
state $s^1\in \aS\setminus C(s)$ is first reached.
We assume the law regenerates at $T^1$, that 
is, at this time the sequence restarts
its evolution with law $\aP_{s^1}$ until time $T^2$ when 
it first reaches $s^2\in \aS\setminus C(s^1)$, and so on. 
The study of stationarity of $\PP$ is made through 
the chain of states $\{s^1, s^2,...\}$ at times $\{T^1, T^2,..\}$. 
In one of our main results, Theorem \ref{statx}, we show that 
$\PP$ is stationary in time if and only if the vector of weights 
$\pi$ is invariant for this chain.

\medskip

We use this result to prove in Theorem \ref{thmren1} of 
Section \ref{secren1} that, when $\pi$ is
invariant and $\{T^1, T^2,..\}$ is aperiodic,  
the probability measure $\PP$ is the 
asymptotic measure of any starting distribution that is
a convex combination of $\aP_s$.

\medskip

We note that stationarity is not a property totally foreign to genomes, in 
fact we show in Section \ref{sec3x} that the well-established 
Chargaff's second parity rule (CSPR) implies stationarity and when 
this law is only assumed to be valid for $k-$mers then
the stationarity holds for cylinders of length 
$k-1$. In genomics, CSPR has been proven to hold in the alphabet of nucleotides 
for $k-$mers of length $k\approx 10$. Then 
stationarity in the alphabet of codons is for length $k\approx 3$.
CSPR was first observed experimentally in
\emph{Bacillus subtilis} \cite{chg2} and confirmed
in sufficiently long sequences for small polymer
chains in \cite{ph}. More recent empirical
studies assessing its validity can be found in \cite{mi},
\cite{hm}, \cite{zh}.

\medskip

In Section \ref{sec1'x} we supply a richer model and give the 
main results of this work. Here, a choice must be made 
at each site where a region of type $C$ encounters a symbol $s'\not\in C$:
either it starts a new region governed by $\aP_{s'}$ or it continues the
former region of type $C$. This decision is modeled by a sequence
of independent random variables in the unit interval and the 
random choice also depends on $s'$.  
The conditions for stationarity of this process are given in 
Theorem \ref{statxR}. In Theorem \ref{thmren2} it is stated 
that the law of the process can be also seen as an asymptotic 
law when starting from an initial weighted distribution.

\medskip

Some of the most relevant works in the statistical analysis 
of DNA sequences have been devoted to describing the
statistical differences between regions of different types. Thus, in
\cite{pb1} and \cite{lk} it is discovered that intergenic sequences have
long-range correlations while 
short-range correlations prevail in genic sequences. An important 
tool constructed in \cite{pb1} was a map from the nucleotide
sequences onto a walk. Then, correlations and other statistical
quantities could be computed in walks and translated to DNA sequences.
These methods were used in \cite{bs} and \cite{bs2} to study stationarity,
where a detailed statistical 
discussion about stationarity or non-stationarity
of genic and intergenic regions can be found, together with an examination of power-type decreasing 
correlation functions. We wish to emphasize that in our model
the laws for the regions $(\aP_s:s\in \aS)$ may have long- or short-range correlations, or neither. They do not need to satisfy any Markovian condition,
there is no need of hidden Markov chains or other kinds of models used in
annotation as in \cite{nb, bp, by} or in references therein. Also, 
these laws do not need to exhibit any kind of
stationarity. We take them as an input to study how they 
can combine and organize together into a unique law that under some 
conditions turns out to be stationary. 

\medskip

We are aware that the models we introduce and study are
far from having the necessary degree of complexity to make it
realistic for describing nucleotide or codon DNA sequences 
in bacterial genomes, but they provide
some insights for their analysis.
Thus, even if the statistical laws of nucleotide or codon sequences 
in genomes are claimed not to be stationary, our results imply that
non-stationarity will not simply arise due to the existence of
two types of regions, genic and intergenic, but from other phenomena
that would contradict our hypotheses.
This could be the case for the regenerative property which is one
of the main ingredients for studying stationarity. 
One might be tempted to think this condition is too strong, but this 
is not so clear because a direct consequence of it, that the sequence of
symbols marking the beginnings of regions is an
homogeneous Markov chain, was shown to hold in annotated bacterial
genomes in recent joint work with A. Hart (\cite{hama}).

\medskip

We point out that there is no intersection, not in any obvious way at least, 
between the probabilistic study carried out in this work 
and the probabilistic studies devoted to genome evolution.
Finally, there is a large bibliography on the statistics of 
codon and nucleotide sequences of bacterial DNA.  Here, we have only cited papers that have a direct relationship to the present study.   For a more complete view of this body of work, the reader is directed to the references contained in those that we have cited.


\section{ A law based on visits}
\label{sec1x}

Our goal is to supply a a global law that mixes, in some natural way, 
distinct probability distributions starting at symbols belonging to a 
defined subset. We will do it similarly to the Kac's 
construction of invariant probability measures for Markov chains. 
In genomics this problem corresponds on how the laws of the genic and 
intergenic regions can be mixed to obtain a global law for the genome. 

\medskip

From now on $I$ denotes a finite alphabet, $\aS$ is 
a nonempty subset of $I$ and its elements are called initial 
symbols of the alphabet. We suppose that $\aS$ is partitioned 
into equivalence classes which define regions of the same type.
We denote by $C(s)$ the class containing $s\in \aS$.

\medskip

Let us introduce some notation and basic concepts.
Every countable set $L$ is endowed with the discrete $\sigma-$field
$\Sets(L)=\{K: K\subseteq L\}$.
We set $\NN=\{0,1,2,..\}$ and $\NN^*=\{1,2,..\}$. 

\medskip

Define $X_n:I^\NN\to I$, $x\to x_n$
to be the $n$-th coordinate function, so $X_n(x)=x_n$ for $x\in I^\NN$.
For each $n\in \NN$,
$$
\B^X_n=\sigma(X_0,..,X_n)
$$
denotes the $\sigma-$field generated by the coordinates $X_0,...,X_n$ and
$$
\B^X_{\infty}=\sigma(X_n: n\in \NN) 
$$
denotes the  $\sigma-$field generated by all the coordinates. 
The product set $I^\NN$ is endowed
with the $\sigma-$field $\B^X_{\infty}$.
For $q\in \NN$, the shift map in $q-$steps of time is 
\begin{equation}
\label{defshift}
\Theta_q:I^\NN\to I^\NN\,, \;\, (\Theta_q x)_n=x_{n+q}
\; \forall n\in \NN\,. 
\end{equation}
The random variables are $\B^X_{\infty}-$measurable functions
$W:I^\NN\to \RR\cup\{-\infty,\infty\}$, so $W(x)$ is the value of 
this variable at $x\in I^\NN$.

\medskip

The set $B\circ \Theta_N^{-1}$ has characteristic
function $\ind_B\circ \Theta_N$ because $x\in B\circ \Theta_N^{-1}$
if and only if $\Theta_N(x)\in B$.

\medskip
 
If $P$ is a probability measure on $(I^\NN,\B^X_{\infty})$ then 
the process $X=(X_n:n\in \NN)$ is said to have distribution $P$. 
When we want to emphasize
the dependence on $P$ we say under (law or distribution) $P$.

\medskip

Let $I^+=\bigcup_{n\in \NN^*} I^n$ be the set of non-empty finite words,
so  $\aS\times I^+=\bigcup_{n\in \NN^*} (\aS\times I^n)$ is the
set of words with length at least two starting with some symbol in $\aS$.

\medskip

Below we use the usual convention $\inf \emptyset = \infty$.  

\medskip

As said $I$ is endowed with the $\sigma$-field 
$\Sets(I)$, and this last class of subsets is endowed with
the $\sigma$-field $\Sets(\Sets(I))$.
Let $\J:I\to  \Sets(I)$, $i\to \J(i)$, be a map. Then,
the function $I^\NN\to \Sets(I)$, $x\to \J(x_0)$ is $\B^X_0-$measurable.
So, $\J(x_0)$ is a random set. 
Let $T_\J$ be the random time to hit $\J$ in the future,
$$
T_\J=\inf\{n> 0: X_n\in \J(X_0)\}\,, \hbox{ so }
T_\J(x)=\inf\{n> 0: x_n\in \J(x_0)\}\,.
$$
It defines the sequence of successive returns to $\J$,
\begin{equation}
\label{succ1}
T^1_\J:=T_\J \hbox{ and } \, 
\forall n\ge 1\,: \quad T^{n+1}_\J=T^{n}_\J+ T_\J\circ \Theta_{T^{n}_\J}\,.
\end{equation}
Here $T^{n}_\J=\infty$ implies $T^{n'}_\J=\infty$ for $n'\ge n$.
Sometimes, the dependence on $X_0$ will be written explicitly 
so we put indistinctly $T_{\J(X_0)}$ or $T_\J$. 

\medskip

Let $(\aP_s: s\in \aS)$ be a family of probability distribution on $I^\NN$. 
Under $\aP_s$ the process $X=(X_n: n\ge 0)$ starts from $s$, so $\aP_s(X_0=s)=1$. 
We denote by $\aE_s$ the expectation defined by $\aP_s$.
We assume the set $\aS\setminus C(s)$ is attained 
in finite time $\aP_s-$a.s.. Hence, when $X_0\in \aS$ we can define the 
random time
$$
T:=T_{\aS\setminus C(X_0)}=\inf\{n>0: X_n\in \aS\setminus C(X_0)\}\,,
$$
which is $\aP_s-$a.s. finite,  
\begin{equation}
\label{imp0}
\forall s\in \aS:\quad \aP_s(T<\infty)=1\,.
\end{equation}
The sequence of successive returns is, 
$$
T^1=T \hbox{ and } 
T^{n+1}=T^{n}+ T\circ \Theta_{T^n} \hbox{ for } n\ge 1\,.
$$
The time $T^{n}$ is called the $n$-th {\it hitting time of a 
different class} and $(T^n:n\in \NN^*)$ is called 
{\it the sequence of hitting times of different classes}.
Note that it is not guaranteed that $T^n$ is finite for $n>1$.
By definition we have 
$$
T^{n+1}<\infty \, \Rightarrow \, C(X_{T^{n+1}})\neq C(X_{T^n}).
$$

A family $(\aP_s: s\in \aS)$ does not determine a common law. 
In the sequel we define a probability measure $\PP$ on $I^\NN$. 
To avoid trivial situations we assume $I\setminus \aS\neq \emptyset$
(as is the case in genomics) because in the contrary 
it will be sufficient 
to weight the laws $(\aP_s: s\in \aS)$ in a simple way.
The probability that  at coordinate $0$ the process 
takes the value $i \in I$, will be obtained by weighting 
the number of visits done to the state $i$ 
by the laws $(\aP_s: s\in \aS)$ previous to hit a region of
a different type. These visits will be weighted with a
strictly positive vector $\pi=(\pi_s: s\in \aS)$,
$\pi_s$ being the weight given to $\aP_s$. 
Even if the distribution we define depends on $\pi$
we shall not explicit it to avoid overburden notation.

\begin{definition}
\label{defmed}
For the family $(\aP_s: s\in \aS)$
and a strictly positive vector $\pi=(\pi_s: s\in \aS)$
we define $\PP$ on $I^\NN$ by:
\begin{eqnarray}
\nonumber
\forall B\in \B^X_\infty:\quad 
\PP(B)&=&\sum\limits_{s\in \aS}\pi_s
\left(\!\sum\limits_{n\ge 0}
\aE_s(\ind_{T> n} \ind_B\circ \Theta_n)\right)\\
\label{def1}
&=&\sum\limits_{s\in \aS}\pi_s
\left(\sum\limits_{n\ge 0}
\aP_s(T> n, B\circ \Theta_n^{-1})\right).
\end{eqnarray}
$\Box$
\end{definition}
\noindent (As usual, in the last expression 
the event $A_1\cap A_2$ is written $(A_1,A_2)$).
Obviously $\PP$ is a measure on $I^\NN$.
Note that for all $(i_l:l=0,..,m)\in I^{m+1}$ we have
\begin{equation}
\label{ca1x}
\PP(X_l\!=\!i_l, l\!=\!0,..,m)=\sum\limits_{s\in \aS}\pi_s 
(\sum\limits_{n\ge 0}
\aP_s(T\!> \!n, X_{l+n}\!=\!i_l, l\!=\!0,..,m)).
\end{equation}
It is useful to develop (\ref{ca1x}) in two different cases. 
For $(i_l:l=0,..,m)\in \aS\times I^m$ we have
\begin{eqnarray}
\nonumber
\PP(X_l=i_l, l\!=\!0,..,m)&=&\sum\limits_{s\in C(i_0)}\!\!\!\!\pi_s
(\sum\limits_{n\ge 0} \aP_s(T\!>\!n\!+\!1, X_{l+n+1}\!=\!i_l, l\!=\!0,..,m))\\
\label{caso1x}
&{}& \;\, + \, \pi_{i_0}\, \aP_{i_0}(X_{l+n}\!=\!i_l, l\!=\!1,..,m).
\end{eqnarray}
For $(i_l: l=0,..,m)\in (I\!\setminus \!\aS)\times I^m$ we get,
\begin{equation}
\label{caso2x}
\PP(X_l\!=\!i_l, l\!=\!0,..,m)=\sum\limits_{s\in \aS}\!\pi_s
(\sum\limits_{n\ge 0}
\aP_s(T\!>\!n\!+\!1,X_{l+n+1}\!=\!i_l, l\!=\!0,..,m)).
\end{equation}

\begin{theorem}
\label{cprobx}
There exists some strictly positive
vector $\pi=(\pi_s: s\in \aS)$ such that the measure
$\PP$ defined by (\ref{ca1x}) is a probability measure
if and only if it is satisfied
\begin{equation}
\label{cexistx}
\forall \, s\in \aS:\;\; \aE_s(T)<\infty \,.
\end{equation}
In this case, the condition on $\pi$:
\begin{equation}
\label{cnormx}
\sum\limits_{s\in \aS}\! \pi_s \aE_s(T)=1\,,
\end{equation}
is necessary and sufficient in order that $\PP$ is 
a probability measure on $I^\NN$. 
\end{theorem}

\begin{proof}
We will show that condition (\ref{cnormx}) is equivalent to
$\PP(X_0\in \aS)=1$. 
Let $s_0\in \aS$. From (\ref{caso1x}) we have
\begin{eqnarray*}
\PP(X_0=s_0)&=& \sum\limits_{s\in C(s_0)} \pi_s (\sum\limits_{n\ge 0}
\aP_s(T\!>\!n, X_{n}=s_0))+\pi_{s_0}\\
&=&\sum\limits_{s\in C(s_0)} \pi_s 
\aE_s(\sum_{n=1}^T\ind_{\{X_{n}=s_0\}})+\pi_{s_0}.
\end{eqnarray*}
Hence
\begin{equation}
\label{sum1}
\PP(X_0\in \aS)=
\sum\limits_{s\in \aS} \pi_s 
\aE_s(\sum_{n=1}^T\ind_{\{X_{n}\in C(s)\}})
+\sum\limits_{s\in \aS} \pi_s.
\end{equation}

On the other hand from (\ref{caso2x}) we obtain 
\begin{eqnarray}
\nonumber 
\PP(X_0\in I \!\setminus \!\aS) &=&
\sum\limits_{s\in \aS}\!
\pi_s (\! \sum\limits_{n\ge 0} \!
\aP_s(T>n+1, X_{n+1}\in I \!\setminus \! \aS))\\
\nonumber
&=&\sum\limits_{s\in \aS}\!
\pi_s (\sum\limits_{n\ge 0} \!
(\aP_s(T\!>\!n+1)-\aP_s(T\!>\!n+1, X_{n+1}\in C(s)))\\
\nonumber
&{}&=\sum\limits_{s\in \aS}\! \pi_s
(\aE_s(T)-\aP_s(T<\infty)-\aE_s
(\sum_{n=1}^T \ind_{\{X_n\in C(s)\}}) )\\
\nonumber
&{}&= \sum\limits_{s\in \aS}\!\pi_s 
(\aE_s(T)-\aE_s (\sum_{n=1}^T \ind_{\{X_n\in C(s)\}})) 
- \sum\limits_{s\in \aS}\! \pi_s \\
\label{sum2}
&{}&=\sum\limits_{s\in \aS} \pi_s \aE_s
(\sum_{n=1}^T\ind_{\{X_{n}\!\in I \setminus \aS\}})
-\sum\limits_{s\in \aS} \pi_s \,.
\end{eqnarray}
From (\ref{sum1}) and (\ref{sum2}) we get,
$$
\sum_{i\in I} \PP(X_0=i)=\sum\limits_{s\in \aS}\! \pi_s \aE_s(T)\,.
$$
Hence, condition (\ref{cnormx}) is necessary and
sufficient in order that $\PP$ is a probability measure on $I^\NN$.
So (\ref{cexistx}) is a necessary and sufficient condition in order 
that there exists such a strictly positive vector $\pi$.
\end{proof}

\medskip

Note that relation (\ref{cnormx}), together with $\pi>0$ and
$\aE_s(T)\ge 1$ for $s\in \aS$, 
imply $\sum_{s\in \aS}\pi_s\le 1$. Moreover 
$\sum_{s\in \aS}\pi_s=1$ if and only if $\aE_s(T)=1$ 
for all $s\in \aS$ which is equivalent to $\aP_s(T=1)=1$
for all $s\in \aS$. In this case the dynamics we study further 
will be trivial. So, we can assume $\pi$ is a 
strictly positive and strictly substochastic vector.

\medskip

From now on we assume (\ref{cexistx}) always hold and that $\pi$
satisfies (\ref{cnormx}), so $\PP$ is a probability measure on $I^\NN$.
We denote by $\EE$ its mean expected value.

\begin{remark}
\label{remfl1}
From (\ref{def1}) and by using (\ref{cnormx}), we get formally
\begin{equation}
\label{formal1}
\PP(T<\infty)= \sum_{s\in \aS} \pi_s(\sum_{n\ge 0} \aP_s(T>n, T<\infty))=
 \sum_{s\in \aS} \pi_s \EE_s(T)=1.
\end{equation}
This is formal because the definition of $T$ requires a pointwise construction 
of $\PP$ where the type of region at the initial time is explicitly
known. This will be done in Section \ref{secren1} for 
a class of laws $\aP_s$ that satisfy a regenerative condition 
and for vectors $\pi$ that define a stationary law $\PP$. $\Box$
\end{remark}

 \section{ Regeneration and conditions for stationarity }
\label{sec2x}

Let us introduce some notation and recall some basic notions. 
For a probability measure $P$ on $I^\NN$, $E$ denotes its 
associated expectation and $E(\cdot \, | \, \B')$
the mean expected value operator with respect to a sub-$\sigma$ field
$\B'\subseteq \B$ and $P$. For $i\in I$ 
we denote by $P_i=P(\cdot \, | \, X_0=i)$ 
the conditional distribution to start from $i\in I$ and by
$E_i$ the expectation associated with $P_i$.

\medskip

A random time $T'$ taking values in 
$\NN\cup \{\infty\}$ is a stopping time
with respect to the filtration $(\B^X_n:n\!\in \!\NN
\cup \{\infty\})$
when $\{T'\le n\}\in \B^X_n$ is satisfied for all $n\in \NN$. 
Its associated $\sigma-$field is 
$$
\B^X_{T'}=\{B\in \B:
B\cap \{T'\le n\}\in \B^X_n, \, \forall n\in \NN\}\,.
$$
It is easy to see that for every $\B(X_0)$ measurable random set 
$\J=\J(X_0)$, the return time $T_\J(X_0)=\inf\{n>0: X_n\in \J(X_0)\}$
is a stopping time. So, for $X_0\in \aS$,  
$T=T_{\aS\setminus C(X_0)}$ is a stopping time. From (\ref{imp0}),
the random time $T$ is finite $\aP_s-$a.s. for all $s\in \aS$. 

\medskip

Let us define a regenerative time in a larger sense than in 
\cite{smi}, Section $3.7$ in \cite{sr} or Chapter V in \cite{as}, 
where it is required  that at such a time the process restarts 
independently as a replica of the initial one. We only need 
that at a regenerative time the process starts in an independent 
way, unique requirement set in \cite{ab}. As said in \cite{mt} 
Section $2.4$, at a regenerative time the strict past is forgotten.

\begin{definition}
\label{def1reg}
Let $T'$ be a stopping time with respect to the filtration 
$(\B^X_n:n\in \NN\cup \{\infty\})$. We say that $P$ regenerates at $T'$ 
if for all bounded measurable function $h:I^\NN\to \RR$ we have
\begin{equation}
\label{regdeb1}
E(\ind_{\{T'<\infty\}}h\circ \Theta_{T'} \, | \, \B_{T'})
=\ind_{\{T'<\infty\}} E_{X_{T'}}(h)\,\; P-\hbox{a.s.}.
\end{equation}
$\Box$
\end{definition}

Let us define a new family of probability measures
$(\aP^*_s: s\in \aS)$ from $(\aP_s: s\in \aS)$ by
regeneration at $T$. To fix it we introduce some new notation. 
For a sequence ${\underline i}=(i_l: l=0,..,m)\in \aS\times I^+$
let $(\tau_n({\underline i}): n\ge 0)$ 
be the set of indexes given by
$$
\tau_0({\underline i})=0 \hbox{ and } \; \forall n\ge 1:\;
\tau_n({\underline i})=\inf\{l> \tau_{n-1}: i_l\in \aS\setminus
C(i_{\tau_{n-1}({\underline i})})\}\,.
$$
Let $\chi({\underline i})=\sup\{n\ge 0: 
\tau_{n}({\underline i})<\infty\}$. From definition,
$$
\forall \, n\in \{1,..,\chi({\underline i})\}:\;\; 
C(i_{\tau_n({\underline i})})
\neq C(i_{\tau_{n-1}({\underline i})})\,.
$$
Let us define the laws $(\aP^*_s: s\in \aS)$. 
Take ${\underline i}=(i_l: l=0,..,m)\in I^+$,
note the functions $\tau_k({\underline i})$ 
by $\tau_k$, but in $\chi({\underline i})$
keep the dependence on ${\underline i}$. We set,
\begin{eqnarray}
\nonumber
\aP^*_s(X_l\!=\!i_l, l\!=\!0,..,m)&=&\ind_{\{i_0=s\}}\!\!
\prod_{k=0}^{\chi({\underline i})-1}\!\!\!\!
\aP_{i_{\tau_k}}(X_l=i_{\tau_k+l}, l=1,..,\tau_{k+1}-\tau_k)\\
\label{caso0x}
&{}&\times \aP_{i_{\tau_{\chi({\underline i})}}}
(X_l=i_{\tau_{\chi({\underline i})}+l},l=1,..,m-
\tau_{\chi({\underline i})}).
\end{eqnarray}
An inductive argument on 
$\chi({\underline i})=0,..,m$ shows that $\aP^*_s$ is
well-defined. Note that $\aP^*_s(X_0=s)=1$ for all $s\in \aS$.
From (\ref{imp0}) we find $\aP_s^*(T<\infty)=1$ for all $s\in \aS$.
Moreover, from definition (\ref{caso0x}), we can
apply Borel-Cantelli lemma to get
\begin{equation}
\label{imp2}
\forall s\in \aS,\, \forall n\in \NN^*:\quad \aP^*_s(T^n<\infty)=1\,.
\end{equation}

We denote by $\aE^*_s$ the mean expected value 
associated with $\aP^*_s$. Note that $\aP_s(B\cap \{T\le n\})=
\aP^*_s(B\cap \{T\le n\})$ for all
$B\in \B_T$ and $n\in \NN$. In particular
$\aP^*_s(T>n)=\aP_s(T>n)$, so $\aE^*_s(T)=\aE_s(T)$.

\begin{proposition}
\label{regdeblem}
For or all probability vector $\gamma=(\gamma_s: s\in \aS)$ 
and all $n\in \NN^*$ 
the distribution $\aP^*_\gamma=\sum_{s\in \aS}\gamma_s \aP^*_s$ 
regenerates at $T^n$. In particular for all $s\in \aS$, $\aP^*_s$ 
regenerates at $T$.
\end{proposition}

\begin{proof}
It suffices to show the statement for $\aP^*_\gamma=\aP^*_s$,
that is for an extremal vector $\gamma$. Also by an inductive
argument it suffices to prove the result for $n=1$, that is for $T^1=T$.
Since $\aP^*_s(T<\infty)=1$, we must show the following 
equality holds for $(j_1,...,j_q)\in I^+$,
\begin{equation}
\label{regdeb}
\aE^*_s(\ind_{\{X_{k+T}\!=\!j_k, k=1,..,q\}} \, | \, \B^X_{T})=
\aE^*_{X_{T}}(\ind_{\{X_k=j_k, k=1,..,q\}}) \;\, \aP^*_s-\hbox{a.s.}.
\end{equation}
Let $(i_0,..i_m)\in I^{+}$ be such that $i_0=s$,
$i_l\in C(s)$ for $l=1,..,m-1$ and $i_m\not\in C(s)$.
Let $B_m=\{T=m, X_l=i_l, l=0,...,m\}$.
Then, (\ref{regdeb}) will be shown once we prove the equality
$$
\int_{B_m} \ind_{\{X_{k+m}\!=\!j_k, k=1,..,q\}} d\aP^*_s=
\int_{B_m} \aP^*_{i_m}(X_k=j_k, k=1,..,q)d\aP^*_s\,.
$$
But, this follows straightforwardly from a recurrence argument
on property (\ref{caso0x}).
\end{proof}

From now on we define the distribution $\PP^*$ 
as in Definition \ref{defmed}
but for the family of probability measures 
$(\aP^*_s: s\in \aS)$ instead
of $(\aP_s: s\in \aS)$. It suffices to replace $\aE_s$ by
$\aE_s^*$ in (\ref{def1}). Since $\aE_s(T)=\aE^*_s(T)$ for
$s\in \aS$, the condition (\ref{cexistx}) is the same and $\pi$
must satisfy the same condition (\ref{cnormx}). Thus, $\PP^*$
is a probability measure on $I^\NN$ and we denote by 
$\EE^*$ its mean expected value.

\medskip

Let $\gamma=(\gamma_s: s\in \aS)$ be a probability
vector and
$\aP^*=\sum_{s\in \aS} \gamma_s \aP^*_s$ be the associated distribution
on $I^\NN$, so $\aP^*(X_0\in \aS)=1$. 
From relation (\ref{caso0x}), $\aP^*$ satisfies 
\begin{eqnarray*}
&{}& \aP^*(X_l\!=\!i_l, l\!=\!0,..,m, T\!=\!m, 
X_{m+k}\!=\!j_k, k\!=\!1,..,t)\\
&=& \aP^*(X_l\!=\!i_l, l\!=\!0,..,m, T\!=\!m)
\aP^*_{i_m}( X_{k}\!=\!j_k, 
k\!=\!1,..,t).
\end{eqnarray*}
Consider the following sequence of variables $(\Xi_n: n\ge 0)$ 
taking values on $\aS$,
$$
\Xi_0=X_0 \hbox{ and } \forall n\ge 1: \;\; \Xi_n=X_{T^n}.
$$ 
By (\ref{imp2}) this is a well defined process. 
Under $\aP^*$, and by using Proposition \ref{regdeblem}, 
we get that the sequence $(\Xi_n: n\in \NN)$ is a Markov 
chain taking values in $\aS$ with transition matrix $Q=(q_{s s'}: s,s'\in \aS)$ 
given by $q_{s s'}=\aP_s(X_{T}=s')$ for $s,s'\in \aS$. 
In fact, from Proposition \ref{regdeblem} we have 
$$
\aP^*(\Xi_{k+1}=s_{k+1}\, | \, 
\Xi_k=s_k,..,\Xi_0=s_0)=\aP^*_{s_k}(X_T=s_{k+1}).
$$
Since $\Xi_{k+1}\in \aS\setminus C(\Xi_k)$, we get that
$q_{s s'}>0$ implies $s'\not\in C(s)$. 

\medskip

Recall that a positive vector $\rho=(\rho_s: s\in \aS)$ is 
invariant for $Q$ if it satisfies the set of equalities
$$
\forall s\in \aS:\;\;
\rho_s= \sum\limits_{s'\in \aS} \rho_{s'} q_{s' s}.
$$
There always exists invariant positive vectors, 
moreover, if $Q$ is irreducible 
an invariant positive vector is unique up 
to a multiplicative constant.

\medskip
 
On the other hand, $\PP^*$ is a stationary distribution on $I^\NN$ 
if for all $m\in \NN$ and all $(i_0,..,i_m)\in I^{m+1}$ we have
\begin{equation}
\label{stat1}
\forall t\ge 1:\;\,
\PP^*(X_{k+t}\!=\!i_k, k\!=\!0,..,m)=
\PP^*(X_{k}\!=\!i_k, k\!=\!0,..,m)\,.
\end{equation}
By an inductive argument, (\ref{stat1}) is satisfied once 
it holds for $t=1$, so stationarity is verified
when for all $m\in \NN$ and all $(i_0,..,i_m)\in I^{m+1}$ it holds
\begin{equation}
\label{stat2x}
\PP^*(X_{k}\!=\!i_l, l\!=\!0,..,m)=\sum\limits_{j\in I}
\PP^*(X_0=j, X_{k+1}\!=\!i_k, k\!=\!0,..,m)\,.
\end{equation}

\begin{theorem}
\label{statx}
Assume that the strictly positive vector 
$\pi=(\pi_s: s\in \aS)$ satisfies the condition
(\ref{cnormx}). Then, $\PP^*$ is stationary if and only if
$\pi$ is invariant for $Q$, that is it satisfies
\begin{equation}
\label{prim3x}
\forall s\in \aS:\;\;
\pi_s= \sum\limits_{s'\in \aS} \pi_{s'} q_{s' s}
\, \hbox{ where } q_{s s'}=\aP_s(X_{T}=s').
\end{equation}
\end{theorem}

\begin{proof}
From (\ref{stat2x}) $\PP^*$ is stationary if for all 
$m\in \NN$ and all $(i_0,..,i_m)\in I^{m+1}$ it is satisfied 
\begin{equation}
\label{stattx}
\PP^*(X_{l}\!=\!i_l, l\!=\!0,..,m)=\sum\limits_{j\in I}
\PP^*(X_0=j, X_{l+1}\!=\!i_l, l\!=\!0,..,m)\,.
\end{equation}
Let,
$$
G=\! \{X_{l}\!=\!i_l, l\!=\!0,..,m\}
\hbox{ and } G\circ \Theta_n^{-1}\!=\! \{X_{l+n}\!=\!i_l, l\!=\!0,..,m\} 
\hbox{ be the }n\!-\!\hbox{shifted set}. 
$$
The stationarity condition is $\PP^*(G)=\PP^*(G\circ \Theta_1^{-1})$.
From (\ref{caso2x}) we obtain
\begin{eqnarray*}
&{}& \PP^*(X_0 \!\in \! I\!\setminus\! \aS, G\circ \Theta_1^{-1})
=\sum\limits_{s\in \aS}\sum\limits_{n\ge 0}
\pi_s \aP^*_s(T\!>\!n\!+\!1,
X_{l+n+1}\!\in \! I\setminus \aS, G\circ \Theta_{n+2}^{-1})\\
&=&\sum\limits_{s\in \aS}\sum\limits_{n\ge 0}\!\!
\pi_s (\aP^*_s(T\!>\!n\!+\!1,
G\circ \Theta_{n+2}^{-1})-\aP^*_s(T\!>\!n\!+\!1, X_{n+1}\!\in \!C(s), 
G\!\circ \Theta_{n+2}^{-1})).
\end{eqnarray*}
From (\ref{caso1x}) we get
\begin{eqnarray*}
\sum\limits_{j\in \aS} \PP^*(X_0=j, G\circ \Theta_{1}^{-1})
&=& \sum\limits_{s\in \aS}\!\!\pi_s (\sum\limits_{n\ge 0}
\aP^*_s(T\!>\!n+1, X_{n+1}\!\in \!C(s), G\circ \Theta_{n+2}^{-1}))\\
&{}&
\;\, +\sum\limits_{s\in \aS}\pi_s \aP^*_s(G\circ \Theta_{1}^{-1}).
\end{eqnarray*}
From the last two expressions
we obtain
\begin{eqnarray}
\nonumber
\sum\limits_{j\in I} \PP^*(X_0=j, G\circ \Theta_{1}^{-1})
&=&
\sum\limits_{s\in \aS}\!\pi_s (\sum\limits_{n\ge 0}
\!\aP^*_s(T\!>\!n\!+\!1,G\circ \Theta_{n+2}^{-1}))\!
+ \sum\limits_{s\in \aS}\! \pi_s \aP^*_s(G\circ \Theta_{1}^{-1})\\
\label{sxq}
&{}&=
\sum\limits_{s\in \aS}\pi_s (\sum\limits_{n\ge 0}\aP^*_s(T\!>\!n+1,
G\circ \Theta_{n+1}^{-1}))\,.
\end{eqnarray}
For $i_0\in I\!\setminus \!\aS$, expression (\ref{caso2x}) implies,
\begin{equation}
\label{pfi11}
\PP^*(G)=\sum\limits_{s\in \aS}\pi_s (\sum\limits_{n\ge 0}
\aP^*_s(T\!>\!n+1, G\circ \Theta_{n+1}^{-1}))\,.
\end{equation}
Hence, (\ref{sxq}) and (\ref{pfi11}) show that, with no additional 
hypothesis, the stationary equality (\ref{stattx}) 
is satisfied when
$(i_0,..,i_m)\in (I\!\setminus \!\aS)\times I^+$.

\medskip

Now, when $i_0\in \aS$, from (\ref{sxq}) we obtain,
\begin{eqnarray}
\nonumber 
\sum\limits_{j\in I} \PP^*(X_0=j, G\circ \Theta_1^{-1})
&=&
\!\!\sum\limits_{s\in C(i_0)}\!\!\!\! \pi_s 
(\sum\limits_{n\ge 0}\aP^*_s(T\!>\!n\!+\!1,
G\circ \Theta_{n+1}^{-1})\\
\label{ig1}
&{}&\;\; +\!\!\!\!\sum\limits_{s\in \aS\setminus C(i_0)} \!\!\!\!\!\!\!
\pi_s (\sum\limits_{n\ge 0}
\aP^*_s(T\!=\!n\!+\!1, G\circ \Theta_{n+1}^{-1}))\!.
\end{eqnarray}

Now we use $\aP^*_s(T<\infty)=1$ for all $s\in \aS$ as well
as the definition done in (\ref{caso0x}), to get
\begin{eqnarray}
\nonumber
&{}&\sum\limits_{s\in \aS\setminus C(i_0)}\!\!\!\!\!\!\pi_s 
(\!\sum\limits_{n\ge 0}\aP^*_s(T\!=\!n\!+\!1, 
G\circ \Theta_{n+1}^{-1}))\\
\nonumber
&{}&
\; =(\sum\limits_{s\in \aS\setminus C(i_0)}\!\!\!\!\!\!
\pi_s (\sum\limits_{n\ge 0}\aP^*_s(T\!=\!n+1, X_T=i_0)))
\cdot \aP^*_{i_0}(G)\\
\label{ig2}
&{}& \; =
(\sum\limits_{s\in \aS\setminus C(i_0)}\!\!\!\!\!
\pi_s \aP^*_s (X_T=i_0)) \cdot \aP^*_{i_0}(G).
\end{eqnarray}
On the other hand by using $i_0\in \aS$ 
formula (\ref{caso1x}) gives
\begin{equation}
\label{ig3}
\PP^*(G)=
\sum\limits_{s\in C(i_0)}\!\!\!\!\pi_s (\sum\limits_{n\ge 0}
\aP^*_s(T\!>\!n+1, G\circ \Theta_{n+1}^{-1}))
+ \pi_{i_0} \aP^*_{i_0}(G).
\end{equation}
Therefore, from equalities (\ref{ig1}), (\ref{ig2}) 
and (\ref{ig3}) we deduce that the equality 
(\ref{stattx}) is satisfied 
if and only if the following relation holds
$$
\forall i_0\in \aS:\quad \pi_{i_0}=
\sum\limits_{s\in \aS\setminus C(i_0)}\!\!\!\!
\pi_s \aP^*_s (X_T=i_0)\,.
$$
Since $\aP^*_s (X_T=s')=0$ when $C(s')=C(s)$, we have proven that 
$\PP^*$ is stationary if and only if the following condition
is satisfied
$$
\forall s'\in \aS:\quad \pi_{s'}=\sum\limits_{s\in \aS}\pi_s 
\aP^*_s (X_T=s')=\sum\limits_{s\in \aS}\pi_s q_{s s'}\,.
$$ 
This shows the theorem.
\end{proof}

When $\PP^*$ is stationary we can extend it to the set of bi-infinite sequences
$I^\ZZ$ by putting
\begin{equation}
\label{biinf1}
\PP^*(X_{l+k}=i_k, k=0,..,m)=\PP^*(X_k=i_k, k=0,..,m)
\end{equation}
for all $l\in \ZZ$, $m\ge 0$ and $(i_k:k=0,..,m)\in I^+$. Note that this 
equality obviously holds for $l\in \NN$ because $\PP^*$ is stationary.

\section{ A renewal property of the law}
\label{secren1}
Define the probability vector
$$
\wpi=(\wpi_s: s\in \aS)
\hbox{ with } \wpi_s=\pi_s\, (\sum_{s'\in \aS}\pi_{s'})^{-1}.
$$
Consider the distribution $\aP_\wpi=\sum_{s\in \aS}\wpi_s 
\aP_s$ on $\B^X_\infty$ and let $\aE_\wpi$ be its mean expected value. 
From (\ref{imp2}) we have $\aP_\wpi(T^n<\infty)=1$ for all $n\in \NN^*$, 
being $T=T_{\aS\setminus C(s)}$. By condition 
(\ref{cnormx}) we also find
$$
\aE_\wpi(T)=(\sum_{s\in \aS}\pi_s)^{-1}
(\sum_{s\in \aS}\pi_s \aE_s(T))=
(\sum_{s\in \aS}\pi_s)^{-1}\,.
$$
Let $\aP^*_\wpi$ be given by 
$\aP^*_\wpi=\sum_{s\in \aS}\wpi_s \aP^*_s$ on 
$\B^X_\infty$ and $\aE^*_\wpi$ be its mean 
expected value. By previous relations
\begin{equation}
\label{esp1}
\forall n\in \NN^* \quad \aP^*_\wpi(T^n<\infty)=1 \hbox{ and } 
\aE^*_\wpi(T)^{-1}=\sum_{s\in \aS}\pi_s\,. 
\end{equation}

We will extend the probability measure spaces 
$(I^\NN, \B^X_\infty)$ where the 
probabilities $\aP^*_s$, $\aP^*_\wpi$, 
$\PP^*$ are defined, to include a 
countable number of independent copies of $X$. 
Since this is a simple 
extension, the probability distributions  
on this space will be continue to be noted 
$\aP^*_s$, $\aP^*_\wpi$, $\PP^*$, respectively.

\medskip

Consider the distribution of $T$: 
$\aP^*_\wpi(T=l)$, $l\in \NN^*$. It is
aperiodic if the greatest common divisor of its support 
satisfies 
\begin{equation}
\label{gcd1}
\hbox{g.c.d.}\{l>0: \aP^*_\wpi(T=l)>0\}=1.
\end{equation}

\begin{theorem}
\label{thmren1}
Assume $\pi$ satisfies (\ref{cnormx}), (\ref{prim3x}) 
and the distribution of $T$ satisfies (\ref{gcd1}). Then, 
\begin{equation}
\label{eqw2}
\forall B\in \B^X_\infty:\quad
\PP^*(B)=\lim\limits_{N\to \infty} \aP^*_\wpi(B\circ \Theta_N^{-1}).
\end{equation}
If in addition the matrix $Q$ is aperiodic then for all
probability vector $\gamma=(\gamma_s: s\in \aS)$ the probability
measure $\aP_\gamma^*=\sum_{s\in \aS} \gamma_s \aP^*_s$ satisfies
\begin{equation}
\label{eqw3}
\forall B\in \B^X_\infty:\quad
\PP^*(B)=\lim\limits_{N\to \infty} \aP_\gamma^*(B\circ \Theta_N^{-1}).
\end{equation}
\end{theorem}

\begin{proof}
Let us first prove the statement (\ref{eqw2}).
It is sufficient to show the equality
for $B=(X_k=i_k, k=0,..,m)$ with $(i_k: k=0,..,m)\in I^+$.

\medskip

Since $\wpi$ is invariant for the stochastic matrix $Q$,
\begin{equation}
\label{eqw}
\forall s\in \aS:\quad  \wpi_s=\aP^*_\wpi(X_{T}=s).
\end{equation}

From Proposition \ref{regdeblem} 
the probability distribution $\aP^*_\wpi$ regenerates 
at times $(T^n: n\in \NN^*)$. An inductive argument on (\ref{eqw}) gives 
$\wpi_s=\aP^*_\wpi(X_{T^n}=s)$, so $X_{T^n}$ is distributed
as $\wpi$. Then, by using (\ref{regdeb1}), for all bounded 
and measurable $h:I^\NN\to \RR$ we have
\begin{equation}
\label{rex1}
\aE^*_\wpi(h\circ \Theta_{T_n} \,| \, \B_{T^n})=
\aE^*_\wpi(h) \quad \aP^*_\wpi-\hbox{a.s.}.
\end{equation}
Under $\aP^*_\wpi$, the increments $(T^1, T^{n+1}-T^n: n\in \NN^*)$ 
are independent equally distributed, each increment having the 
same distribution as $T$. For proving  (\ref{eqw2}) it is useful 
to give a renewal construction of $\aP^*_\wpi$. 

\medskip

Consider a sequence $(X^{(m)}: m\in  \NN^*)$ of 
independent copies of the process 
$(X_n: n\le T)$ with distribution $\aP^*_\wpi$, so 
$X^{(m)}=(X^{(m)}_n: n\le T^{[m]})$ where 
$(T^{[m]}: m\in  \NN^*)$ is a
sequence of independent copies of $T$. 
Define the process $\wX=(\wX_n: n\ge 0)$ by
$$
\wX_n=X^{(m)}_{n'} \hbox{ if } n=\sum_{l=1}^{m-1} T^{[l]}+n' \hbox{ and 
} n'\le T^{[m]}. 
$$
Since $\aP^*_\wpi$ regenerate 
at $(T^n: n\in \NN)$ and $X_{T^n}$has distribution $\wpi$, 
we get that $\wX$ and $X$ are equally distributed 
with distribution $\aP^*_\wpi$. 

\medskip

Let $\wT^{n}=\sum_{k=1}^n T^{[m]}$ and $\wT=\wT^1$. Then, 
$(\wT^{n}: n\in \NN^*)$ is the sequence of hitting times of 
different classes of $\wX$ and it has 
the same distribution as $(T^n: n\in \NN^*)$.
By construction $(T^{[m]}: m\in \NN^*)$ are independent identically 
distributed random variables with common distribution 
$\aP^*_\wpi(T^{[m]}=l)=\aP^*_\wpi(\wT=l)$ for $l\in \NN^*$. 
Since these variables have finite mean 
$\aE^*_\wpi(\wT)=(\sum_{s\in \aS}\pi_s)^{-1}$, we can apply 
the renewal theorem, see  Chapter II in \cite{tl}. Define,
$$
\forall\, N>0:\quad \beta_N=\sup\{\wT^n: \wT^n\le N, n\in \NN^*\},
$$
where we put $\beta_N=0$ if $\wT^n> N$ for all $n\in \NN^*$.
Note that the distribution of $\beta_N$ only depends on
the sequence $(\wT^n:n\ge 1)$. By aperiodicity 
of the distribution of $\wT$ the renewal theorem gives: 
\begin{equation}
\label{kx1}
\forall l\ge 0: \;\, \lim\limits_{N\to \infty}
\aP^*_\wpi({\beta_N}\!=\!N\!-\!l)=
E_\wpi(T)^{-1}\aP^*_\wpi(\wT\!>\!l)=
(\sum_{s\in \aS}\pi_s)\aP^*_\wpi(\wT\!>\!l)
\end{equation}
and,
\begin{equation}
\label{dis2}
\lim\limits_{N\to \infty}
\aP^*_\wpi(N\in \{\wT^n: n\in \NN^*\})=
\aE^*_\wpi(\wT)^{-1}=\sum_{s\in \aS}\pi_s.
\end{equation}
We have
$$
\aP^*_\wpi(\wX_{N+k}=i_k, k=0,..,m)=
\sum_{l=0}^N\aP^*_\wpi(\wX_{N+k}=i_k, k=0,..,m, \beta_N=N-l).
$$
Let $\epsilon>0$ and $r>0$. From (\ref{kx1}) we obtain 
\begin{equation}
\label{ench1}
\exists\, N'=N'(\epsilon,r)
\hbox{ such that } \forall \, N\ge N':\;\;
\aP^*_\wpi(\beta_N< N-r)<\epsilon.
\end{equation}
Hence
$$
|\aP^*_\wpi(\wX_{N+k}=i_k, k=0,..,m)-\!\!
\sum_{l=0}^r \aP^*_\wpi(\wX_{N+k}\!=\!i_k, k=0,..,m, 
\beta_N\!=\!N\!-l)|<\epsilon.
$$
By regeneration at times $\{\wT^n: n\in \NN^*\}$ 
(see (\ref{rex1})) we get 
\begin{eqnarray}
\nonumber
&{}&\aP^*_\wpi(\wX_{N+k}=i_k, k=0,..,m; \beta_N=N-l)\\
\nonumber
&{}&=\aP^*_\wpi(\wX_{N+k}\!=\!i_k, k\!=\!0,..,m; 
N\!-\!l\in \{\wT^n: n\!\in \!\NN^*\}, \\
\nonumber
&{}&\quad\quad \;\; (N\!-\!l,N]
\cap \{\wT^n: n\!\in \!\NN^*\}\!=\!\emptyset)\\
\label{ax1}
&{}&=\aP^*_\wpi(\wX_{l+k}=i_k, k=0,..,m,\wT>l)
\aP^*_\wpi(N-l\in \{\wT^n: n\in \NN^*\}).
\end{eqnarray}
From (\ref{dis2}) we get the existence of $N''(r,\epsilon)>N'$ 
such that for all $N>N''$ we have
\begin{equation}
\label{ax2}
\forall \, l\!\in \! \{0,..,r\}: \quad
\big|\aP^*_\wpi(N\!-\!l\in \{\wT^n: n\!\in \!\NN^*\})-
(\sum_{s\in \aS}\pi_s)\big|<\frac{\epsilon}{r}.
\end{equation}
Therefore
\begin{equation}
\label{ax3}
|\aP^*_\wpi(\wX_{N+k}\!=\!i_k, k\!=\!0,..,m)-
\sum_{l=0}^r \!
\aP^*_\wpi(\wX_{l+k}\!=\!i_k, k\!=\!0,..,m;T\!>\!l)
(\sum_{s\in \aS}\pi_s)|\!<\!2\epsilon.
\end{equation}
Since $X$ and $\wX$ are equally distributed we get
\begin{eqnarray}
\nonumber
&{}&
\lim\limits_{N\to \infty}\aP^*_\wpi(X_{N+k}=i_k, k=0,..,m)\\
\nonumber
&{}&=\sum_{l=0}^\infty \aP^*_\wpi(T>l, X_{l+k}=i_k, k=0,..,m)
(\sum_{s\in \aS} \pi_s)\\
\nonumber
&{}&
=\sum_{l=0}^\infty (\sum_{s\in \aS} \pi_s
\aP^*_s(T>l, X_{l+k}=i_k, k=0,..,m))\\
\label{ax4}
&{}& =\PP^*(X_{k}=i_k, k=0,..,m).
\end{eqnarray}
Then, the proof of (\ref{eqw2}) for $B=\{X_k=i_k, k=0,..,m\}$
is finished.

\medskip

Let us now show (\ref{eqw3}). Since $Q$ is aperiodic,  
$\wpi$ is the unique invariant probability measure for $Q$ and 
every probability vector $\nu=(\nu_s: s\in \aS)$ satisfies
\begin{equation}
\label{maap}
\lim\limits_{N\to \infty}\nu' Q^N=\wpi'.
\end{equation}
where we note by $\nu'$ the row vector, transpose of $\nu$.
From Proposition \ref{regdeblem}
the probability distribution $\aP^*_\gamma$ regenerates at 
times $(T^n: n\in \NN)$ (they are finite $\aP^*_\gamma$-a.s.). 
Denote by $\gamma(n)$ the distribution of $X_{T^n}$ on $\aS$ 
when we start from $\aP^*_\gamma$. Since $X_{T^{n-1}}$ 
is distributed as $\gamma(n\!-\!1)$ and
$\aP^*_{\gamma(n-1)}(X_T=s')=(\gamma(n)' Q)_{s'}$ we get 
$\gamma(n)'=\gamma' Q^{n}$. So, (\ref{maap}) gives 
\begin{equation}
\label{maap1}
\lim\limits_{N\to \infty}\gamma(n)=\wpi.
\end{equation}

Note that for all event $D\in \B_\infty^X$ and all 
probability vector $\nu$ we have
\begin{equation}
\label{maap8}
|\aP^*_\nu(D)-\aP^*_\wpi(D)|=
|\sum_{s\in \aS} (\nu_s-\wpi_s)\aP^*_s(D)|
\le \sum_{s\in \aS}|\nu_s-\wpi_s|.
\end{equation}
Let us fix $\epsilon>0$.
Since $\aP_\wpi(D)=\sum_{s\in \aS}\aP_s(D)$, 
when $\aP_\wpi(D)<\epsilon$ we get $\aP_s(D)<\epsilon/\wpi_s$. 
Since $\wpi>0$, from (\ref{ench1}) we obtain that 
for all $\epsilon>0$ and $r>0$,
\begin{equation}
\label{maap2}
\exists {\underline N}(\epsilon,r)
\hbox{ such that } \forall s\in \aS, \; 
\forall N\ge {\underline N}(\epsilon,r):
\;\, \aP^*_s(\beta_N< N-r)<\epsilon.
\end{equation}
Define the sequence of random variables $(\eta_N: n\in \NN^*)$
by
$$
\eta_N=\sup\{n\in \NN^*: T^n\le N\},
$$
where we put $\eta_N=0$ if $T^n> N$ for all $n\in \NN^*$.
The sequence $(\eta_N: N\in \NN^*)$ is increasing and
$$
\forall \; s\in \aS:\quad \aP_s^*
(\lim\limits_{N\to \infty}\eta_N=\infty)=1. 
$$
Then, 
$$
\forall {\tilde r}\in \NN^* \;\, \exists \,
{\underline N}'({\tilde r},\epsilon) \;
\forall  N\ge {\underline N}'({\tilde r},\epsilon) 
\; \forall s\in \aS:\;\, 
\aP_s^*(\eta_N\le {\tilde r})<\epsilon.
$$
Hence, for all $N\ge {\underline N}'({\tilde r},\epsilon)$
we have $\aP_\gamma^*(\eta_N\le {\tilde r})<\epsilon$.
This last relation and (\ref{maap1}) implies the existence 
of ${\underline N}''(\epsilon)$ that satisfies
$$
\forall  N\ge {\underline N}''(\epsilon):\quad
\aE_\gamma^*(\sum_{s\in \aS}|\gamma(\eta_N)_s-\wpi_s|)<\epsilon.
$$
Then, above relation and (\ref{maap8}) implies
that for all event $D\in \B_\infty^X$ and
$N\ge {\underline N}''(\epsilon)$ it is satisfied
\begin{eqnarray}
\nonumber
|\aE^*_\gamma(\aP^*_{\gamma(\eta_N)}(D))-\aP^*_\wpi(D)| 
&=&|\aE^*_\gamma(\aP^*_{\gamma(\eta_N)}(D)-\aP^*_\wpi(D))|\\
\label{maap10}
&\le& \aE^*_\gamma(\sum_{s\in \aS}|\gamma(\eta_N)_s-\wpi_s|)
<\epsilon.
\end{eqnarray}
From (\ref{maap2}) we get that for all
$N\ge {\underline N}(\epsilon,r)$ it holds
$$
|\aP^*_\gamma(X_{N+k}\!=\!i_k, k\!=\!0,..,m)-\!\!
\sum_{l=0}^r \!
\aP^*_\gamma(X_{N+k}\!=\!i_k, k\!=\!0,..,m, \beta_N\!=\!N\!-\!l)|
\!<\!\epsilon.
$$

By regeneration at times $\{T^n: n\in \NN^*\}$,
see (\ref{regdeb1}), and since the law
of $X_{T^n}$ is $\gamma(n)$ we obtain for all $l=0,..,r$:
\begin{eqnarray*}
&{}&\aP^*_\gamma(X_{N+k}=i_k, k=0,..,m; \beta_N=N-l)\\
&{}&=\aP^*_\gamma(X_{N+k}\!=\!i_k, k\!=\!0,..,m;
N\!-\!l \!\in \!\{T^n: n\!\in \!\NN^*\}, \\
&{}&\quad\quad\;\;\; (N\!-\!l,N]\cap \{T^n: n\!\in \!\NN^*\}\!=\!\emptyset)\\
&{}&=\aE_\gamma^*(\aP^*_{\gamma(\eta_{N-l})}
(X_{l+k}\!=\!i_k, k=0,..,m; T>l))
\aP^*_\gamma(N-l\in \{T^n: n\in \NN^*\}).
\end{eqnarray*}

From (\ref{maap10}) we get that all
$N\ge {\underline N}''(\epsilon)+r$ satisfies
$$
\big|\aE_\gamma^*(\aP^*_{\gamma(\eta_{N-l})}
(X_{l+k}\!=\!i_k, k\!=\!0,..,m; T\!>\!l))-
\aP^*_{\wpi}(X_{l+k}\!=\!i_k, k\!=\!0,..,m;T\!>\!l)\big|<\epsilon.
$$
From (\ref{ax1}), (\ref{ax2}), (\ref{ax3}) and (\ref{ax4}), we get
that the proof will be complete once we show
\begin{equation}
\label{maap5}
\lim\limits_{N\to \infty}\aP^*_\gamma(N\in \{T^n: n\in \NN^*\})=
\lim\limits_{N\to \infty}\aP^*_\wpi(N\in \{T^n: n\in \NN^*\}).
\end{equation}
From (\ref{maap2}) we get for all $N\ge {\underline N}(\epsilon,r)$ 
$$
\aP^*_\wpi(\beta_N\!<\! N\!-\!r)+
\aP^*_\gamma(\beta_N\!<\! N\!-\!r)<2\epsilon.
$$
Then, for all $k>0$ we obtain
\begin{eqnarray*}
&{}& 
\big|\aP^*_\gamma(N\!+\!k\!\in \!\{T^n: n\!\in \! \NN^*\})
-\aP^*_\wpi(N\!+\!k\!\in \! \{T^n: n\!\in \! \NN^*\})\big|
\\
&{}& \le
\big|\aP^*_\gamma(N\!+\!k\!\in\! \{T^n: n\!\in \!\NN^* \}, 
\beta_N\!\ge \!N\!-\!r)-
\aP^*_\wpi(N\!+\!k\!\in\! \{T^n: n\!\in \! \NN^*\},
\beta_N\!\ge \! N\!-\!r)\big|\\
&{}& \;\;\;\, + \,2\epsilon.
\end{eqnarray*}
We have
\begin{eqnarray*}
&{}& \big|\aP^*_\wpi(N\!+\!k \!\in \! \{T^n: n\!\in\! \NN^*\},
\beta_N\!\ge \! N\!-\!r)-\aP^*_\gamma(N\!+\! k \!\in \!
\{T^n: n\!\in \!\NN^*\},\beta_N\!\ge \!N\!-\!r)\big|\\
&{}&=
\sum_{l=0}^r \, (\aP^*_\wpi(\beta_N\!=\!N\!-\!l)
\aP^*_\wpi(k\!+\!l \!\in\! \{T^n: n\!\in\!\NN^*\})\\
&{}&\quad\quad\quad -\aP^*_\gamma(\beta_N\!=\!N\!-\!l)
\aE^*_\gamma(\aP^*_{\gamma(\eta_{N-l})}
(k\!+\!l\!\in\! \{T^n: n\!\in\!\NN^*\}))).
\end{eqnarray*}
For all $N\ge {\underline N}''(\epsilon)$ we have
$$
|\aE^*_\gamma(\aP^*_{\gamma(\eta_N)}
(k\!+\!l\!\in\! \{T^n: n\!\in\!\NN^*\}))-
\aP^*_\wpi(k\!+\!l\!\in\! \{T^n: n\!\in\!\NN^*\})|<\epsilon.
$$
We use $\aP^*_\wpi(k\!+\!l\!\in\! \{T^n: n\!\in\!\NN^*\})
=(\sum_{s\in \aS}\pi_s)^{-1}$ and put all these relations together
to conclude,
\begin{eqnarray*}
&{}& \big|\aP^*_\wpi(N\!+\!k \!\in \! 
\{\wT^n: n\!\in\!\NN^*\},\beta_N\ge N\!-\!r)
-\aP^*_\gamma(N\!+\!k \!\in \! \{\wT^n: n\!\in \!\NN^*\},
\beta_N\ge N\!-\!r)\big|\\
&{}& \le
\big|\sum_{l=0}^r (\aP^*_\gamma(\beta_N=N\!-\!l)-
\aP^*_\wpi(\beta_N=N\!-\!l))
\big|(\sum_{s\in \aS}\pi_s)^{-1}+\epsilon\\
&{}& \le 2\epsilon(\sum_{s\in \aS}\pi_s)^{-1}. 
\end{eqnarray*}
Hence (\ref{maap5}) is shown. Therefore, relation (\ref{eqw3}) is proven.
\end{proof}

\begin{remark}
\label{rem1x}
Note that in the proof of property (\ref{eqw3}) 
we require a starting measure 
of the type $\sum_{s\in \aS} \gamma_s \aP_s^*$ because, 
on the one hand the definition of $T$ needs that the
starting state in $\aS$ is defined and, on the other hand 
we use the regenerative equality of this measure as stated in 
(\ref{caso0x}). $\Box$
\end{remark}

Let $\PP^*$ the be law defined on $I^\ZZ$ by (\ref{biinf1}). The same 
proof showing property (\ref{eqw2}) in Theorem \ref{thmren1} 
allows us to prove that for all $l\in \ZZ$ and all 
$(i_k: k=0,..,m)\in I^+$ we have
\begin{equation}
\label{eq11}
\PP^*(X_{l+k}=i_k, k=0,..,m)=
\lim\limits_{N\to \infty} \aP_\wpi^*(X_{l+k+N}=i_k, k=0,..,m).
\end{equation}

From (\ref{maap2}) we deduce that the random variable 
$$
T^0=\sup\{T^n: T\le 0\}
$$ 
which takes value in $\{l\in \ZZ: l\le 0\}$), 
has a proper distribution. That is $T^0$
is finite $\PP^*-$a.s.. Since there is regeneration at $T_0$ 
the random variable $X_{T_0}$ has distribution $\wpi$.
Then,
\begin{eqnarray}
\nonumber
&{}& \PP^*(X_0=i_k, k=0,..,m)=\sum_{n\in \NN}
\PP^*(T^0=-n, X_0=i_k, k=0,..,m)\\
\label{eq12}
&{}&\; =\sum_{n\in \NN} \PP^*(X_0=i_k, k=0,..,m \, | \, T^0=-n) \PP^*(T^0=-n).
\end{eqnarray}
Since $\PP^*$ regenerates at each $T^n$ with law $\wpi$, we have  
\begin{equation}
\label{eq13}
\PP^*(X_0=i_k, k=0,..,m \, | \, T^0=-n)=
\aP_\wpi(X_n=i_{k+n}, k=0,..,m \, | \, T>n).
\end{equation}
On the other hand from (\ref{kx1}) we get,
\begin{equation}
\label{eq14}
\aP^*_\wpi(T^0=-n)=
(\sum_{s\in \aS}\pi_s)\aP^*_\wpi(T>n).
\end{equation}
Then, we retrieve the definition done in (\ref{def1}),
$$
\PP^*(X_0=i_k, k=0,..,m)=
\sum_{s\in \aS} \sum_{n\in \NN} \pi_s \aP^*_s(X_n=i_{k+n}, k=0,..,m;  T>n).
$$
Hence (\ref{eq12}), (\ref{eq13}) and (\ref{eq14}) give a probabilistic 
insight to definition $\PP^*$ and allow us to have a good definition of $T$
under law $\PP^*$, 
as claimed in Remark \ref{remfl1}. 

\section{Stationarity and Chargaff second parity rule}
\label{sec3x}

Let $\Li$ be an alphabet and $Y_n:\Li^\NN\to \Li$
be the $n$-th coordinate function: $Y_n(y)=y_n$ for $y\in L^\NN$.
Let $\varphi:\Li\to \Li$ be a convolution, this means
$\varphi$ is one-to-one and $\varphi^{-1}=\varphi$.
Since $\varphi$ is a bijection we have 
$L=\{\varphi(h): h\in \Li\}$.

\medskip

Let $P$ be a probability measure on $\Li^\NN$. 
We say that $P$ satisfies the 
Chargaff second parity rule (CSPR) with respect to $\varphi$ if
for all $m\in \NN$, all $(l_0,..,l_m)\in \Li^{m+1}$
and all $t\in \NN$ it is satisfied:
\begin{equation}
\label{chff1}
P(Y_{k+t}=l_k, k=0,..,m)=
P(Y_{k+t}=\varphi(l_{m-k}), k=0,..,m)\,.
\end{equation}
We claim that (\ref{chff1}) is satisfied if it holds
for $t=0$. That is, if for all 
$m\in \NN$ and all $(l_0,..,l_m)\in \Li^{m+1}$,
\begin{equation}
\label{chff2}
P(Y_k=l_k, k=0,..,m)=P(Y_k=\varphi(l_{m-k}), k=0,..,m)\,.
\end{equation}
In fact, from (\ref{chff2}) we get,
\begin{eqnarray*}
&{}&P(Y_k=h_k, k\!=\!0,..,t\!-\!1; Y_{t+k}\!=\!l_k,
k\!=\!0,..,m; Y_{t+m+k}\!=\!c_k, k\!=\!0,..,t\!-\!1)\\
&{}&
=P(Y_k=\varphi(c_{t-1-k}), k=0,..,t\!-\!1;
Y_{k+t}=\varphi(l_{m-k}), k=0,..,m; \\
&{}& \quad\quad\quad\quad\quad\quad\quad
\quad\quad\quad\quad\quad\quad\quad\quad\;
Y_{k+t+k}=\varphi(h_{t-1-k}), k=0,..,t\!-\!1).
\end{eqnarray*}
Hence, by summing on $(h_0,..,h_{t-1})\in \Li^t$ and 
$(c_0,..,c_{t-1})\in \Li^t$ we get (\ref{chff1}).

\begin{proposition}
\label{chimst}
If $P$ verifies the CSPR then it is stationary.
\end{proposition}

\begin{proof}
Assume $P$ satisfies the CSPR. 
For all $m\in \NN^*$ we have
\begin{eqnarray*}
&{}& \sum_{h\in \Li}P(Y_0=h, Y_{k+1}=l_{k}, k=0,..,m)\\
&{}&=
\sum_{h\in \Li}P(Y_{m+1}=\varphi(h), Y_{m-k}=
\varphi(l_{k}), k=0,..,m)\\
&{}&
=P(Y_{m+1}\in \Li, Y_{m-k}=\varphi(l_{k}), k=0,..,m)\\
&{}&
=P(Y_{m-k}=\varphi(l_{k}), k=0,..,m)=
P(Y_{k}=l_{k}, k=0,..,m)\,.
\end{eqnarray*}
Then, the result follows.
\end{proof}

Let us fix $d\in \NN^*$ and consider $I:=\Li^d$ as a new alphabet.
Take the following transformation $\zeta:\Li^\NN\to I^\NN$,
$y\to x=\zeta y$ with $x_n=(\zeta y)_n=(y_{dn},..,y_{d(n+1)-1})$.
Let $P\circ \zeta^{-1}$ be the induced law on $I^\NN$.
We claim that if $P$ is stationary, then also
$P\circ \zeta^{-1}$  is stationary. Let
$X_n:I^\NN\to I$ be the $n-$th coordinate function, we must prove
that for all $m\in \NN$ and
$((l_{dk},..,l_{d(k+1)-1}):k=0,..,m)\in I^{m+1}$ we have
\begin{eqnarray*}
&{}&
P\circ \zeta^{-1}(X_k=(l_{dk},..,l_{d(k+1)-1}), k=0,..,m)\\
=&{}& \!\!
\sum_{(c_0,..,c_{d-1})\in \Li^r}\!\!\!\!\!\!\!\!
P\circ \zeta^{-1}(X_0=(c_0,..,c_{d-1}),
X_{k+1}=(l_{dk},..,l_{d(k+1)-1}, k=0,..,m)\,.
\end{eqnarray*}
This relation is equivalent to,
\begin{eqnarray*}
&{}& P(Y_t=l_t, t=0,..,dm-1)\\
&{}& =
\sum_{(c_0,..c_{d-1})\in \Li^d}\!\!\!\!\!
P(Y_0=c_0,..,Y_{d-1}=c_{d-1}; Y_{t+d}=l_t, t=0,..,dm-1)\,,
\end{eqnarray*}
which is equivalent to the equality
$$
P(Y_t=l_t, t=0,..,dm-1)=P(Y_{t+d}=l_{t}, t=0,..,dm-1)\,.
$$
This last relation follows straightforward from the stationarity 
of $P$, proving that $P\circ \zeta^{-1}$ is stationary.

\medskip

For $I=\Li^d$, let $\psi:I\to \A$ be an onto-function
and consider the function $\Psi:\Li^\NN\to \A^\NN$, 
$y\to \Psi y$ by $(\Psi y)_n=\psi((\zeta y)_n)$. 
We claim that if $P$ is stationary,
then also $P\circ \Psi^{-1}$  is stationary. Denote
by $Z_n:\A^\NN\to \A$ the $n-$th coordinate function, 
we must show that
$$
P\circ \Psi^{-1}(Z_k=a_k, k=0,..,m)=
\sum_{b\in I}P\circ \Psi^{-1}(X_0=b, X_{k+1}=a_{k}, k=0,..,m)\,.
$$
From the equality,
\begin{eqnarray*}
&{}&
P\circ \Psi^{-1}(Z_k=a_k, k=0,..,m)\\
&{}&=\!\!\!\!
\sum_{(l_{dk},..,l_{d(k+1)-1})\in \psi^{-1}\{a_k\}, 
k=0,..,m}\!\!\!\!\!\!\!
\!\!\!\!\!\!\!\!\!
P \circ \zeta^{-1}(X_k=(l_{dk},..,l_{d(k+1)-1}), k=0,..,m),
\end{eqnarray*}
the stationarity of $P\circ \Psi^{-1}$ is retrieved from the
stationarity of $P\circ \zeta^{-1}$.

\medskip

Let us state that a weaker condition of CSPR implies 
a weaker stationary property. Assume that the CSPR 
is verified only for words of length smaller or equal to
$t$. This means that for all for all $m<t$, 
all $(l_0,..,l_m)\in \Li^{m+1}$
and all $u\in \NN$ it is satisfied:
$$
P(Y_{k+u}=l_k, k=0,..,m)=
P(Y_{k+u}=\varphi(l_{m-k}), k=0,..,m)\,.
$$
Let us prove that in this case the stationarity only holds 
for cylinders of length strictly smaller than $t$. 

\begin{proposition}
\label{chimstxl}
Let $t\ge 2$. Assume $P$ verifies the CSPR for cylinders 
defined by words of length smaller or equal to $t$, then
for all $m<t-1$ and all $(l_0,..,l_m)\in \Li^{m+1}$ we have
$$
\forall u\ge 1:\;\,
P(Y_{k+u}\!=\!l_k, k\!=\!0,..,m)=
P(Y_{k}\!=\!l_k, k\!=\!0,..,m)\,.
$$
\end{proposition}

\begin{proof}
We prove it by induction on $u\ge 1$. For $u=1$ the proof is the
same as the one done in Proposition \ref{chimst}.
Assume it has been shown up to $u$, let us prove it for $u+1$.
Since $m+2\le t$ we get,
\begin{eqnarray*}
&{}& P(Y_{u+1+k}=l_{k}, k=0,..,m)=
\sum_{h\in \Li}P(Y_{u}=h, Y_{u+1+k}=l_{k}, k=0,..,m)\\
&{}&=
\sum_{h\in \Li}P(Y_{u+m-k}=\varphi(l_{k}), 
k=0,..,m; Y_{u+1+m}=\varphi(h))\\
&{}&
=P(Y_{u+m-k}=\varphi(l_{k}), k=0,..,m)=
P(Y_{u+k}=l_{k}, k=0,..,m)\,.
\end{eqnarray*}
Then, an inductive argument is applied to get
$P(Y_{u+1+k}=l_{k}, k=0,..,m)=P(Y_{k}=l_{k}, k=0,..,m)$.
Hence, the result follows.
\end{proof}

\medskip

In a genomic framework $\Li=\{A,C,G,T\}$  and
$\varphi:\{A,C,G,T\}\to \{A,C,G,T\}$ is the
involution given by $\varphi(A)=T$, $\varphi(C)=G$.
The CSPR is satisfied for the empirical probability measure 
on the DNA nucleotide sequence of bacterial genome.
So, $d=3$, $I=\Li^3$ is the list of codons
and $\PP=P\circ \zeta^{-1}$ on $I^\NN$. The alphabet 
$\A$ is the list of aminoacids and
$\psi:\Li^3\to \A$ is the genetic code. Hence, a
consequence of CSPR is that the probability distribution
on the nucleotide sequences, and so on codon sequences, 
is stationary. If one accepts that 
CSPR is only valid for small $t-$mers of nucleotides 
with $t\approx 10$, then the weak stationarity property 
stated in Proposition \ref{chimstxl} implies stationarity 
for $9-$mers in the nucleotide sequence,
which in the alphabet of codons means stationarity 
for triplets of codons.

\medskip

For a fundamental explanation of CSPR, it is argued in \cite{hmo} 
that it would be a probabilistic consequence of the reverse
complementarity between paired strands, because symmetry
of chemical energy implies Gibbs distribution is
invariant by reverse complementarity which is exactly CSPR.

\section{Random model}
\label{sec1'x}

We will modify the model studied in Sections 
\ref{sec1x}, \ref{sec2x} and \ref{secren1} 
to approach some of the phenomena occurring in codon 
sequences of bacteria genome. 
Up to now  a region of a new type starts when a state of
a different class is hit. Nevertheless, it is known 
in genome annotation that when an intergenic region hits 
a start codon of a genic region only a small proportion 
of these start codons mark the beginning of a genic region. 
Some signals must be present in the neighborhood 
of the codon to trigger a beginning. 
Nowadays, there is an active research on this domain, 
either on the list of motifs and on the localization they 
must be with respect to the starting codons. A recent discussion 
on this topic can be found in \cite{rw}. 

\medskip

So, at each time a site containing a state of a different class is hit 
a decision must be made: either a new region starts, 
or this beginning is postponed and continues to be governed by  
the symbol of the former region. We will model this decision 
by a random choice, in this purpose we use a 
sequence of independent random variables 
uniformly distributed in the unit interval. Our model admits that
the decision depends on the hit state. 
 
\medskip

From now on we assume that each symbol $s\in \aS$ has 
a probability $\epsilon(s)\in (0,1]$ of start governing a new 
region when it is hit by a region of type different from $C(s)$.
Note that $\epsilon(s)> 0$ is a natural constraint, 
in fact in the contrary we could delete $s$ from $\aS$. 
The case $\epsilon(s)=1$ means that it is sure that when a 
region of type $C$, with $C\neq C(s)$, encounters a site
containing a symbol $s$ then a region governed by $\aP_s$ starts. 

\medskip

The sample region for the random choice is the unit interval 
$R:=[0,1]$ which is endowed with the Borel $\sigma-$field
$\B(R)$ and the Lebesgue measure denoted by $|\cdot|$. So $|R'|$ 
is the Lebesgue measure of a Borel set $R'\subseteq R$. 
The product spaces $R^\NN$ and $R^{\NN^*}$ are respectively 
endowed with the Borel product $\sigma-$fields
noted respectively by $\B(R^\NN)$ and $\B(R^{\NN^*})$.
By $e$ me denote a random variable 
uniformly distributed in $R$ and $P^{e}$ denotes this distribution. 
Let ${\vec e}=(e_m: n\in \NN^*)$ be a sequence of independent
random variables uniformly distributed in $R$.
So, the distribution of ${\vec e}$ 
in $R^{\NN^*}$ is the product measure, 
$P^{\vec e}:=P^{e\otimes \NN^*}$. Consider the projection 
$Z_n:R^\NN\to R$, $z\in R^\NN\to z_n\in R$
for $n\in \NN$. For $R'_1,..R'_m\in \B(R)$ we have
$$
P^{\vec e}(Z_k\in R'_k, k=1,..,m)=\prod_{k=1}^m |R'_k|\,.
$$

Define the space $\K:=I\times R$, whose points 
are pairs $(i,r)\in I\times R$. 

\medskip

To each $s\in \aS$ we associate
a fixed interval $R_s\subseteq R$ with size 
$|R_s|=\epsilon(s)$. 

\medskip

We consider the following dynamics: if 
a region governed by $s$ encounters a pair $(s',r')$ 
then a new region starts if and only if 
$C(s')\neq C(s)$ and $r'\in R_{s'}$.
In this new setting, the set of starting states is
$$
\aV=\bigcup_{s\in \aS} \{s\}\times R_s \,.
$$ 
Hence $(\aS\times R)\setminus \aV=
\bigcup_{s\in \aS} \{s\}\times (R\setminus R_s)$ 
are the states having a starting symbol  
but with a value in the sample region that prevent 
it to start governing a new region.

\medskip

The class $\C(v)$ associated with $v=(r,s)\in \aV$ is defined 
to be
$$
\C(v)=\bigcup_{s'\in C(s)}\{s'\}\times R_{s'}.
$$
Hence $(\C(s,r)=\C(s',r')) \Leftrightarrow 
(C(s)=C(s'), r\in R_s, r'\in R_{s'})$.

\medskip

For all $s\in \aS$ we define the conditional law 
$P^{e_0}(\cdot \,| \, s)$ to be 
uniformly distributed on $R_s$, that is
$$
P^{e_0}(R' \,| \, s)=|R'\cap R_s|/|R_s|\,,\;\, R'\in \B(R)\,.
$$ 

Consider the product space 
$\K^\NN=(I\times R)^\NN=I^\NN\times R^\NN$. We set 
$\X_n:\K^\NN\to \K$, $w\in \K^\NN\to \X(w)=w_n\in \K$ 
the projection onto the $n-$th component. 
Let $w_n=(x_n,r_n)$, we denote 
$X_n:\K^\NN\to I$, $w\to x_n$ and $Z_n:\K^\NN\to R$, $w\to r_n$.
So, we can write $\X_n=(X_n,Z_n)$. This is an abuse of notation 
with $X_n$ and $Z_n$, in fact we will also continue writing
$X_n:I^\NN\to I$, $x\in I^\NN\to x_n\in I$ 
and $Z_n:R^\NN\to R$, $z\in R^\NN\to z_n\in R$. 
We keep the same notation for the shift
$\Theta_q:\K^\NN\to \K^\NN\,, \;\, (\Theta_q w)_n=w_{n+q}$, 
as the one introduced in (\ref{defshift}) for $I^\NN$ and 
use the same notation $\Theta_q$ for the shift in  $R^\NN$.

\medskip

As before we endow $\K^\NN$ with the $\sigma-$field 
$\B^\X_{\infty}=\sigma(\X_n: n\in \NN)$ and we denote
$\B^\X_n=\sigma(\X_0,..,\X_n)$.
Let $P$ be a probability measure on $(\K^\NN, \B^\X_{\infty})$.
A random time $T':\K^\NN\to \NN\cup \{\infty\}$ 
is a stopping time with respect to the filtration 
$(\B^\X_n:n\in \NN)$ 
when $\{T'\le n\}\in \B^\X_n$ is satisfied for all $n\in \NN$. The
$\sigma-$field associated to a stopping time $T'$ was 
already defined and denoted by $\B^\X_{T'}$. 

\medskip

Assume $X_0\in \aV$. Then, the random time
$$
\T:=T_{\aV\setminus \C(\X_0)}=\inf\{n>0: \C(\X_n)\neq \C(\X_0)\},
$$
is well-defined (it can take the value $\infty$) and it is 
a stopping time. As already done for a random time 
in (\ref{succ1}), we define the sequence  of times
$$
\T^1=\T \hbox{ and for } n\in \NN^*: \;\;
\T^{n+1}=\T^n+\T\circ \Theta_{\T^n},
$$
which are also stopping times. We have $\T^1=\T$ and 
$$
\T^{n+1} \hbox{ finite implies }  \, 
C(\X_{\T^{n+1}})\neq \C(\X_{\T^n}).
$$

Let $(\aP_s: s\in \aS)$ be a family of probability 
distribution on $I^\NN$ such that for all 
$s\in \aS$, $\aP_s(X_0=s)=1$ and satisfies condition (\ref{imp0}). 
Each $\aP_s$ defines the following probability measure 
$\aP_s^{\dag}$ on $\K^\NN$: 
for all $m\in \NN$, $(i_0,..,i_m)\in I^{m+1}$ and 
$R'_0,..,R'_m\in \B(R)$,
\begin{eqnarray*}
\nonumber
&{}& \aP_s^{\dag}(X_k=i_k, Z_k\in R'_k, k=0,..,m)\\
&{}&=\ind_{\{i_0=s\}} P^{e_0}(R'_0  | s)
\aP_s(X_k\!=\!i_k,k\!=\!1,..,m)
P^{\vec e}(Z_k\in R'_k, k\!=\!1,..,m).
\end{eqnarray*}
Then, the initial distribution of $\aP_s^{\dag}$ is the 
uniform one on $\{s\}\times R_s$. Note that
\begin{equation}
\label{inic1}
\forall R'\in \B(R), \; 
R'\supseteq R_s: \quad \aP_s^{\dag}(X_0=s, Z_0\in R')=1,
\end{equation}
in particular 
$\aP_s^{\dag}(X_0=s, Z_0\in R)=\aP_s^{\dag}(X_0=s)=1$.

\medskip

Since condition (\ref{imp0}) ensures  
$\aS\setminus C(s)$ is attained in finite time $\aP_s-$a.s., we
apply the  Borel-Cantelli Lemma to the independent 
random variables $(Z_n: n\in \NN^*)$ to get
\begin{equation}
\label{imp0r}
\forall s\in \aS:\;\; \aP^{\dag}_s(\T<\infty)=1\,.
\end{equation}
Let $\aE_s^{\dag}$ be the expected value 
defined by $\aP_s^{\dag}$.

\medskip

The following definition will depend on a 
strictly positive vector vector 
$\pi^{\dag}=(\pi^{\dag}_s: v\in \aS)$.

\begin{definition}
\label{defr}
For the family $(\aP^{\dag}_s: s\in \aS)$ and  
$\pi^{\dag}=(\pi^{\dag}_s: s\in \aS)>0$ 
we define $\PP^{\dag}$ on $\K^\NN$ by: 
\begin{eqnarray}
\nonumber
\forall B\in \B^\X_\infty:\quad 
\PP^{\dag}(B)&=&\sum\limits_{s\in \aS}\pi^{\dag}_s
\left(\!\sum\limits_{n\ge 0}\aE^{\dag}_s(\ind_{\T> n} 
\ind_B\circ \Theta_n)\right)\\
\label{eqx1}
&=&\sum\limits_{s\in \aS}\pi^{\dag}_s
\left(\sum\limits_{n\ge 0}\aP^{\dag}_s(\T> n,
B\circ \Theta_n^{-1})\right),
\end{eqnarray}
where $\Theta_n$ is the shift operator on $\K^\NN$.
\end{definition}

Obviously $\PP^{\dag}$ is a measure. For all $m\in \NN^*$, 
$(i_0,..,i_m)\in I^{m+1}$ and $R'_0,..,R'_m\in \B(R)$ we have
\begin{eqnarray}
\nonumber
&{}& \PP^{\dag}(X_k=i_k, Z_k\in R'_k, k=0,..,m)\\
\label{ca1xR}
&{}& =\sum\limits_{s\in \aS}\pi^{\dag}_s (\sum\limits_{n\ge 0}
\aP^{\dag}_s(\T\!> \!n, X_{k+n}\!=\! i_l,  
Z_{k+n}\in R'_k, k\!=\!0,..,m)).
\end{eqnarray}

\begin{theorem}
\label{cprobxR}
There exists some strictly positive vector 
$\pi^{\dag}=(\pi^{\dag}_s: s\in \aS)$ 
such that $\PP^{\dag}$ defined by (\ref{ca1xR}) is a 
probability measure if and only if it is satisfied
\begin{equation}
\label{cexistxR}
\forall \, s\in \aS:\;\; \aE^{\dag}_s(\T)<\infty \,.
\end{equation}
In this case, the condition on $\pi^{\dag}$
\begin{equation}
\label{cnormxR}
\sum\limits_{s\in \aS}\! \pi^{\dag}_s \aE_s^{\dag}(\T)=1\,,
\end{equation}
is necessary and sufficient in order that $\PP^{\dag}$ 
is a probability measure on $I^\NN$.
\end{theorem}

\begin{proof}
We must show that condition (\ref{cnormxR}) is equivalent to 
$\PP^{\dag}(X_0\in \aS, Z_0\in R)=1$. 
Let $s_0\in \aS$. From (\ref{inic1}) we get
\begin{eqnarray*}
\PP^{\dag}(X_0=s_0, Z_0\in R)&=&\sum\limits_{s\in \aS} 
\!\pi^{\dag}_s (\!\sum\limits_{n\ge 0}
\aP_s^{\dag}(\T\!>\!n\!+\!1, X_{n+1}=s_0, Z_0\in R))\!+
\! \pi^{\dag}_{s_0}\\
&=& \sum\limits_{s\in \aS}\! 
\pi^{\dag}_s (\sum\limits_{n\ge 0}
\aP_s^{\dag}(\T\!>\!n\!+\!1, X_{n+1}\!=\!s_0))\!+\! 
\pi^{\dag}_{s_0}\\
&=&\sum\limits_{s\in \aS} \!
\pi^{\dag}_s \, \aE_s^{\dag}(\!\sum_{n=1}^\T\ind_{\{X_{n}\!=
\!s_0\}}) \!+\! \pi^{\dag}_{s_0}.
\end{eqnarray*}
Hence
\begin{equation}
\label{sum1R}
\PP^{\dag}(X_0\in \aS, Z_0\in R)=
\sum\limits_{s\in \aS} \pi^{\dag}_s 
\aE_s^{\dag}(\sum_{n=1}^\T\ind_{\{X_{n}\in \aS\}})
+\sum\limits_{s\in \aS} \pi^{\dag}_s.
\end{equation}
On the other hand 
\begin{eqnarray*}
\nonumber
\PP^{\dag}(X_0\in I \!\setminus \!\aS, Z_0\in R)&=&
\sum\limits_{s\in \aS}\!
\pi^{\dag}_s (\sum\limits_{n\ge 0} 
\aP_s^{\dag}(\T>n\!+\!1, X_{n+1}\in I \!\setminus \! \aS, 
Z_{n+1}\in R))\\
\nonumber
&{}&\;\;=\sum\limits_{s\in \aS}\!
\pi^{\dag}_s \, (\sum\limits_{n\ge 0} \!
\aP_s^{\dag}(\T>n\!+\!1, X_{n+1}\in I \!\setminus \! \aS)).
\end{eqnarray*}
We have
$$
\aP_s^{\dag}(\T>n+1, X_{n+1}\in I \!\setminus \! \aS)=
\aP_s^{\dag}(\T\!>\!n+1)-
\aP_s^{\dag}(\T\!>\!n\!+\!1, \X_{n+1}\!\in \! 
(\aS\times R)\setminus \aV)
$$
and
\begin{eqnarray*}
\aP_s^{\dag}(\T\!>\!n\!+\!1, \X_{n+1}\!\in \!
(\aS\times R)\setminus \aV)
&=&\aE_s^{\dag}(\ind_{\{\X_{n+1}\in (\aS\times R)
\setminus \aV, \T>n+1\}})\\
&=&\aE_s^{\dag}(\ind_{\{X_{n+1}\in \aS, \T >n+1\}}).
\end{eqnarray*}
Then
\begin{eqnarray*}
\nonumber
&{}&\PP^{\dag}(X_0\! \in  \!I \!\setminus \!\aS, Z_0\! \in \!R)
=\sum\limits_{s\in \aS}\! \pi_s (\sum\limits_{n\ge 0} 
\aP_s^{\dag}(\T\!>\!n+1)-\aE^{\dag}_s(\ind_{X_{n+1}\in \aS, \T>n+1})))\\
\nonumber
&{}&=\sum\limits_{s\in \aS}\! \pi^{\dag}_s
(\aE_s^{\dag}(\T)-\aP_s^{\dag}(\T<\infty)-
\aE_s^{\dag}(\sum_{n=1}^\T \ind_{\{X_n\in \aS\}})).
\end{eqnarray*}
We conclude
\begin{equation}
\label{sum2R}
\PP^{\dag}(X_0\!\in \!I \!\setminus \!\aS, Z_0\!\in \!R)=
\sum\limits_{s\in \aS}\!\!\pi^{\dag}_s
(\aE_s^{\dag}(\T)-\aE_s^{\dag}
(\sum_{n=1}^\T \ind_{\{X_n\in \aS\}}))- 
\sum\limits_{s\in \aS}\! \pi^{\dag}_s.
\end{equation}
From (\ref{sum1R}) and (\ref{sum2R}) we find,
$$
\PP^{\dag}(X_0\in \aS, Z_0\in R)=
\sum\limits_{s\in \aS}\! \pi^{\dag}_s 
\aE_s^{\dag}(\T)\,.
$$
Hence, condition (\ref{cnormxR}) is necessary and sufficient 
in order that $\PP^{\dag}$ is a probability measure on $\K^\NN$.
So (\ref{cexistxR}) is a necessary and sufficient 
condition in order that there exists some strictly positive 
vector $\pi^{\dag}$ fulfilling (\ref{cnormxR}).
\end{proof}

From now on we assume (\ref{cexistxR}) always 
hold and that $\pi^{\dag}$ satisfies (\ref{cnormxR}), 
so $\PP^{\dag}$ is a probability measure on $\K^\NN$.
Let $\EE^{\dag}$ be its associated mean expected value. 

\begin{remark}
\label{remfl2}
From (\ref{defr}) and by using (\ref{cnormxR}), we get formally
\begin{equation}
\label{formal1}
\PP^{\dag}(\T<\infty)= \sum_{s\in \aS} \pi^{\dag}_s(\sum_{n\ge 0} 
\aP^{\dag}_s(\T>n, \T<\infty))=
 \sum_{s\in \aS} \pi_s \EE^{\dag}_s(T)=1.
\end{equation}
Similar comments to those of Remark \ref{remfl1} can be made. $\Box$
\end{remark}

Let us define the family of probability measures 
$(\aP^{* \dag}_s: s\in \aS)$ on $\K^\NN$ by 
regeneration at $\T=T_{\aV\setminus \C(\X_0)}$. In this purpose
for ${\underline{\kappa}}=(\kappa_l: 
l=0,..,m)\in \aV\times \K^m$ define the sequence of indexes 
$(\tau_n({\underline{\kappa}}): n\ge 0)$ by
$$
\tau_0({\underline{\kappa}})=0 \hbox{ and } \, 
\forall n\ge 1\,: \;\;
\tau_n({\underline{\kappa}})=\inf\{l> \tau_{n-1}: 
\kappa_l\in \aV\setminus
\C(\kappa_{\tau_{n-1}({\underline{\kappa}})})\}\,.
$$
Let $\chi({\underline{\kappa}})=
\sup\{k\ge 0: \tau_{k}({\underline{\kappa}})<\infty\}$. 
From definition,
$$
\forall \, n\in \{1,..,\chi({\underline{\kappa}})\}:\;\; 
\C(\kappa_{\tau_n({\underline{\kappa}})})\neq 
\C(\kappa_{\tau_{n-1}({\underline{\kappa}})})\,.
$$
Let us simply note $\tau_k({\underline{\kappa}})$ 
by $\tau_k$, but in $\chi({\underline{\kappa}})$ 
we keep the dependence on ${\underline{\kappa}}$.
Let us define $\aP^{* \dag}_s$. For
$m\in \NN^*$, $(i_0,..,i_m)\in I^{m+1}$ 
and $R'_0,..,R'_m\in \B(R)$ we put
\begin{eqnarray}
\nonumber
&{}&\aP^{* \dag}_s(X_l\!=\!i_l, Z_l\in R'_l, l\!=\!0,..,m)\\
\nonumber
&{}&\, =
\ind_{\{i_0=s\}}P^{e_0}( R'_0 | \, s)
\!\!\!\!\prod_{k=0}^{\chi({\underline{\kappa}})-1}\!\!\!\!
\aP^{* \dag}_{i_{\tau_k}}(X_l=i_{\tau_k+l}, Z_l\!\in \!R'_{\tau_k+l}, 
l=1,..,\tau_{k+1}\!-\!\tau_k)\\
\label{caso0xr}
&{}&\;\; \times \aP^{* \dag}_{i_{\tau_{\chi({\underline{\kappa}})}}}
(X_l=i_{\tau_{\chi({\underline{\kappa}})}+l}, 
Z_l\!\in \!R'_{\tau_{\chi({\underline{\kappa}})}+l},
l=1,..,m\!-\!\tau_{\chi({\underline{\kappa}})}).
\end{eqnarray}

An inductive argument on $\chi({\underline{\kappa}})$ shows that 
$\aP^{* \dag}_s$ is well-defined by (\ref{caso0xr}).
From definition, $\aP^{* \dag}_s(X_0=s)=1$ 
for all $s\in \aS$. 

\medskip

Let us fix $\T=\T_{\aV\setminus \C(\X_0)}$. 
From definition (\ref{caso0xr}), we can
apply Borel-Cantelli lemma to get
\begin{equation}
\label{impr2}
\forall s\in \aS,\, \forall n\in \NN^*:\quad \aP^{* \dag}_s(\T^n<\infty)=1\,.
\end{equation}

We denote by $\aE^{* \dag}_s$ the mean expected value
associated with $\aP^{* \dag}_s$. Note that $\aP^\dag_s(B\cap \{T\le n\})=
\aP^{* \dag}_s(B\cap \{T\le n\})$ for all
$B\in \B_T$ and $n\in \NN$. In particular
$\aP^{* \dag}_s(T>n)=\aP^\dag_s(T>n)$, so $\aE^{* \dag}_s(T)=\aE^\dag_s(T)$.

\medskip

Similarly to Proposition \ref{regdeblem} we can state
the regeneration property.

\begin{proposition}
\label{reg3}
For all probability vector $\gamma=(\gamma_s: s\in \aS)$
and all $n\in \NN^*$
the distribution $\aP^{* \dag}_\gamma=\sum_{s\in \aS}\gamma_s \aP^{* \dag}_s$
regenerates at $\T^n$. In particular for all $s\in \aS$, $\aP^{* \dag}_s$
regenerates at $\T$.
\end{proposition}

\begin{proof}
It suffices to show the statement for $\aP^{* \dag}_\gamma=\aP^{* \dag}_s$,
that is for an extremal vector $\gamma$. Also by an inductive
argument it suffices to prove the result for $n=1$, that is for $\T^1=\T$.
Since $\aP^{* \dag}_s(\T<\infty)=1$,
we must show the following equality for $j_k\in I$,  
$R'_k\in \B(R)$, $k=1,,..,q$:
\begin{eqnarray}
\nonumber
&{}& \aE^{* \dag}_s(\ind_{\{X_{k+\T}=j_k, Z_{k+\T}\in R'_k; k=1,..,q\}} 
\, | \, \B^\X_{\T})\\
\label{regdx}
&{}& = 
\aE^{* \dag}_{X_\T}(\ind_{\{X_k=j_k, Z_k\in R'_k; k=1,..,q\}}) \;\;
\aP^{* \dag}_s-\hbox{a.s.}.
\end{eqnarray}
Let $i_l\in I, R''_l\in \B(R)$, $l=0,..,m$, be such that $i_0=s$,
$i_l\in C(s)$ for $l=1,..,m-1$ and $i_m\not\in C(s)$; and
$R''_0\subseteq R_s$, $R''_m\subseteq R_{i_m}$.
Let $B_m=\{\T=m, X_l=i_l, Z_l\in R''_l, l=0,...,m\}$.
Then, (\ref{regdx}) will be shown once we prove the equality
$$
\int_{B_m}\!\!\!\!\! \ind_{\{X_{k+m}=j_k, Z_{k+m}\in R'_k; 
k=1,..,q\}} d\aP^{* \dag}_s=
\int_{B_m} \!\!\!\!\! \aP^*_{i_m}(X_k=j_k, Z_{k}\in R'_k; 
k=1,..,q)d\aP^{* \dag}_s\,.
$$
This follows from a recurrence argument on (\ref{caso0xr}).
\end{proof}

\medskip

We define $\PP^{* \dag}$ for the family 
$(\aP^{* \dag}_s: s\in \aS)$
simply by putting  $\aE^{* \dag}_s$  instead of $\aE^{\dag}_s$
in  Definition \ref{defr}. Since
$\aE^{* \dag}_s(\T)=\aE^{\dag}_s(\T)$ the condition
(\ref{cexistxR}) supplied by Theorem \ref{cprobxR} in 
order that $\PP^{\dag}$
is a probability measure is the same as for $\PP^{* \dag}$,
that is the vector $\pi^{\dag}$ must satisfy (\ref{cnormxR}).
Let $\EE^{* \dag}$ be the mean expected value associated
with $\PP^{* \dag}$. 

\medskip

Assume $\X_0\in \aV$. Let $\gamma=(\gamma_s: s\in \aS)$ 
be a probability vector 
on $\aS$ and let $\K^\NN$ be endowed with the distribution 
$\aP^{* \dag}=\sum_{s\in \aS} \gamma_s \aP^{* \dag}_s$,
so $\aP^{* \dag}(X_0\in \aS)=1$. By
(\ref{impr2}) the times $(\T^n: n\in \NN^*)$ are finite $\aP^{* \dag}-$a.s.. 
Define the sequence $(\Xi_n: n\ge 0)$ by $\Xi_0=X_0$ and $\Xi_n=X_{\T^n}$
for $n\ge 1$. Proposition \ref{reg3} implies that the sequence
$(\Xi_n: n\in \NN)$ is a Markov chain. The transition matrix
$Q^{\dag}=(q^{\dag}_{s s'}: s,s'\in \aS)$ of this chain is given by
$$
\forall s,s'\in \aS:\quad 
q^{\dag}_{s s'}=\aP^{* \dag}_s(X_{\T}=s').
$$
By definition of $\T$ we have $C(\Xi_{k+1})\neq C(\Xi_k)$, 
so $q^{\dag}_{s s'}>0$ implies  $C(s')\neq C(s)$. 

\medskip

A positive vector $(\rho=(\rho_s: s\in \aS)$ is invariant for $Q^{\dag}$
if it verifies the set of equalities
$$
\forall s\in \aS:\;\;
\rho_s= \sum\limits_{s'\in \aS} \rho_{s'} q^{\dag}_{s' s}.
$$
There always exist invariant positive vectors. 
Moreover, if $Q^{\dag}$ is irreducible the invariant positive vectors
are unique up to a multiplicative constant.
In a similar way as we did in Theorem \ref{statx} we can state 
the following condition for stationarity of $\PP^{* \dag}$.

\begin{theorem}
\label{statxR}
Assume that the strictly positive vector
$\pi^{\dag}=(\pi^{\dag}_s: s\in \aS)$ satisfies the condition
(\ref{cnormxR}). Then, $\PP^{* \dag}$ is stationary 
if and only if $\pi^{\dag}$ is an invariant vector for 
$Q^{\dag}$, that is it satisfies
\begin{equation}
\label{primxR3}
\forall s\in \aS:\;\;
\pi^{\dag}_s= \sum\limits_{s'\in \aS} \pi^{\dag}_{s'} q^{\dag}_{s' s}
\, \hbox{ where } q^{\dag}_{s' s}=\aP^{\dag}_s(X_{\T}=s').
\end{equation}
\end{theorem}

\begin{proof}
We have that $\PP^{* \dag}$ is stationary if for all
$m\in \NN$ and all $i_l\in I$, $R'_l\in \B(R)$, $l=0,..,m$, it is satisfied
\begin{equation}
\label{stattxR}
\PP^{* \dag}(X_{l}\!=\!i_l, Z_l\!\in \! R'_l, l\!=\!0,..,m)=
\PP^{* \dag}(X_{l+1}\!=\!i_l, Z_{l+1}\!\in \! R'_l, l\!=\!0,..,m).
\end{equation}
From now on we denote 
$$
G=\{X_{l}\!=\!i_l, Z_l\!\in \! R'_l, l\!=\!0,..,m\}
$$
and the $n-$shifted set 
$$
G\circ \Theta_n^{-1}=\{X_{l+n}\!=\!i_l, Z_{l+n}\!\in \! R'_l, l\!=\!0,..,m\}.
$$
So, the relation (\ref{stattxR}) that we want to show is
$\PP^{* \dag}(G)=\PP^{* \dag}(G\circ \Theta_1^{-1})$. We have
\begin{eqnarray*}
\PP^{* \dag}(G\circ \Theta_1^{-1})
&=&\PP^{* \dag}(X_0\!\in \! I\!\setminus \!\aS, G\circ \Theta_1^{-1})
+\PP^{* \dag}(X_0\!\in \! \aS, Z_0\!\not\in \!R_{X_0}, G\circ \Theta_1^{-1})\\
&{}&\;\,+ \PP^{* \dag}(X_0\!\in \! \aS, Z_0\!\in \! 
R_{X_0}, G\circ \Theta_1^{-1}).
\end{eqnarray*}
Now
\begin{equation}
\nonumber
\PP^{* \dag}(X_0\!\in \!I \! \setminus \! \aS, G\circ \Theta_1^{-1})\\
=\sum\limits_{s\in \aS}\! \pi^{\dag}_s (\!\sum\limits_{n\ge 0}
\! \aP^{* \dag}_s(\T\!>\!n\!+\!1,
X_{n+1}\!\in \!I\!\setminus \!\aS, G\circ \Theta_{n+2}^{-1})
\end{equation}
and
\begin{eqnarray*}
\nonumber\!
&{}&\PP^{* \dag}
(X_0\!\in \! \aS, Z_0\!\not\in \!R_{X_0}, G\circ \Theta_1^{-1})\\
&{}&\, =\sum\limits_{s\in \aS}\! \pi^{\dag}_s (\!\sum\limits_{n\ge 0}\!
\aP^{* \dag}_s(\T\!>\!n\!+\!1, X_{n+1}\!\in \! \aS, Z_{n+1}\! \not\in \!R_{X_{n+1}},
G\circ \Theta_{n+2}^{-1})\\
&{}&\, =\sum\limits_{s\in \aS}\! \pi^{\dag}_s (\!\sum\limits_{n\ge 0}\!
\aP^{* \dag}_s(\T\!>\!n\!+\!1, X_{n+1}\!\in \! C(s), Z_{n+1}\! \not\in \!R_{X_{n+1}}, 
G\circ \Theta_{n+2}^{-1})\\
&{}&\;\, +\sum\limits_{s\in \aS}\! \pi^{\dag}_s (\!\sum\limits_{n\ge 0}\!
\aP^{* \dag}_s(\T\!>\!n\!+\!1, X_{n+1}\!\in \! \aS\!\setminus \!C(s),
Z_{n+1}\! \not\in \!R_{X_{n+1}},G\circ \Theta_{n+2}^{-1})). 
\end{eqnarray*}
Now, by using $\aP_s^{* \dag}(X_0\!=\!s, Z_0\!\in \! R_{X_0})=1$ we find
\begin{eqnarray}
\nonumber
&{}& \PP^{* \dag}(X_0\!\in \! \aS, Z_0\!\in \! R_{X_0}, G\circ \Theta_1^{-1})\\
\nonumber
&{}&\, =\sum\limits_{s\in \aS}\!\pi^{\dag}_s (\!\sum\limits_{n\ge 0}\!
\aP_s^{* \dag}(\T\!>\!n\!+\!1, X_{n+1}\!\in \! C(s), Z_{n+1}\!\in \! R_{X_{n+1}}, 
G\circ \Theta_{n+2}^{-1}))\\
\nonumber
&{}&\;\;+\sum\limits_{s\in \aS}\!\pi^{\dag}_s (\!\sum\limits_{n\ge 0}\!
\aP_s^{* \dag}(\T\!=\!n\!+\!1, X_{n+1}\!\in \! \aS\!\setminus \!C(s), 
Z_{n+1}\!\in \! R_{X_{n+1}}, G\circ \Theta_{n+2}^{-1}))\\
\nonumber
&{}&\;\;+\sum\limits_{s\in \aS}\pi^{\dag}_s \aP_s^{* \dag}(G\circ \Theta_1^{-1}). 
\end{eqnarray}
On the other hand we have
\begin{eqnarray}
\nonumber
&{}& \aP^{* \dag}_s(\T\!>\!n\!+\!1, G\circ \Theta_{n+2}^{-1})\\
\nonumber
&{}&=\aP^{* \dag}_s(\T\!>\!n\!+\!1,
X_{n+1}\!\in \!I\!\setminus \!\aS, G\circ \Theta_{n+2}^{-1})
+\aP^{* \dag}_s(\T\!>\!n\!+\!1,
X_{n+1}\!\in \! C(s), G\circ \Theta_{n+2}^{-1})\\
\nonumber
&{}& \;\;\; +\aP^{* \dag}_s(\T\!>\!n\!+\!1,
X_{n+1}\!\in \! \aS\!\setminus \!C(s), Z_{n+1}\!\not\in \!R_{X_{n+1}}, 
G\circ \Theta_{n+2}^{-1}).
\end{eqnarray}
We put the previous elements together to get
\begin{eqnarray}
\nonumber
\PP^{* \dag}(G\circ \Theta_1^{-1})
&=&\sum\limits_{s\in \aS}\!\pi^{\dag}_s (\!\sum\limits_{n\ge 0}
\aP^{* \dag}_s(\T\!>\!n\!+\!1,G\circ \Theta_{n+2}^{-1}))+
\sum\limits_{s\in \aS}\pi^{\dag}_s \aP_s^{* \dag}(G\circ \Theta_1^{-1})\\
\nonumber
&{}&\; +
\sum\limits_{s\in \aS}\!\pi^{\dag}_s (\sum\limits_{n\ge 0}
\aP_s^{* \dag}(\T\!=\!n\!+\!1, X_{n+1}\!\in \! \aS\!\setminus \!C(s),
Z_{n+1}\!\in \! R_{X_{n+1}}, G\circ \Theta_{n+2}^{-1}).
\end{eqnarray}
Hence
\begin{eqnarray}
\nonumber
\PP^{* \dag}(G\circ \Theta_1^{-1})
&=&\sum\limits_{s\in \aS}\pi^{\dag}_s (\sum\limits_{n\ge 0}
\aP^{* \dag}_s(\T\!>\!n\!+\!1,G\circ \Theta_{n+1}^{-1}))\\
\label{tri6}
&{}& \; +
\sum\limits_{s\in \aS}(\!\sum\limits_{n\ge 0}\pi^{\dag}_s 
\aP_s^{* \dag}(\T\!= \!n\!+\!1, G\circ \Theta_{n+1}^{-1})).
\end{eqnarray}
Recall (\ref{eqx1}), 
$$
\PP^{* \dag}(G)
=\sum\limits_{s\in \aS}\pi^{\dag}_s (\sum\limits_{n\ge 0}
\aP^{* \dag}_s(\T\!\ge\!n,G\circ \Theta_{n}^{-1})), 
$$
and $G=\{X_{l}\!=\!i_l, Z_l\!\in \! R'_l, l\!=\!0,..,m\}$. 
In both cases: $i_0\in I\!\setminus \!\aS$, or $i_0\in \aS$ and 
$R'_0\subseteq R\setminus R_{i_0}$; we get
$\aP^{* \dag}_s(G)=0$ and $\aP_s^{* \dag}
(\T\!= \!n\!+\!1, G\circ \Theta_{n+1}^{-1})=0$ for all $n\ge 0$. 
Then, in these cases, the stationarity property 
$\, \PP^{* \dag}(G)=\PP^{* \dag}(G\circ \Theta_1^{-1})$,
is an straightforward consequence of formulae (\ref{tri6}) and
(\ref{eqx1}).

\medskip 

We are left to study the case  $i_0\in \aS$ and $R'_0\subseteq  R_{i_0}$.
In this case 
$$
\sum\limits_{s\in \aS}\!\pi^{\dag}_s \aP^{* \dag}_s(G)=
\pi^{\dag}_{i_0}\frac{|R'_0|}{|R_{i_0}|}
\aP^{* \dag}_{i_0}(X_l\!=\!i_l, Z_l\!\in \! R'_l, l\!=\!1,..,m).
$$
On the other hand, since $\aP^{* \dag}_s$ is 
defined by using the regeneration property (\ref{caso0xr}) we get
\begin {eqnarray*}
&{}& \sum\limits_{s\in \aS}(\sum\limits_{n\ge 0}\aP_s^{* \dag}
\left(\T\!= \!n\!+\!1, G\circ \Theta_{n+1}^{-1}\right))\\
&{}&=(\sum\limits_{n\ge 0}\aP_s^{* \dag}(\T\!= \!n\!+\!1, X_{\T}=i_0, 
R_{\T}\in R'_0))\cdot
\aP^{* \dag}_{i_0}(X_{l}\!=\!i_l, Z_l\!\in \! R'_l, l\!=\!1,..,m).
\end{eqnarray*}
We have
$$
\sum\limits_{s\in \aS}\!\pi^{\dag}_s (\sum\limits_{n\ge 0}
\aP_{s}^{* \dag}(\T\!= \!n\!+\!1, X_{\T}=i_0,R_{\T}\in R'_0))=
\frac{|R'_0|}{|R_{i_0}|}(\sum\limits_{s\in \aS}\!\pi^{\dag}_s 
\PP^{* \dag}(X_{\T}=i_0)).
$$
Therefore, from  (\ref{tri6}) and (\ref{eqx1}) we obtain the equivalence
$$
\left(\PP^{* \dag}(G)=\PP^{* \dag}(G\circ \Theta_1^{-1})\right)\Leftrightarrow
(\forall i_0\in \aS: \;
\pi^{\dag}_{i_0}=\sum\limits_{s\in \aS}\!\pi^{\dag}_s 
\PP^{* \dag}(X_{\T}=i_0)).
$$
We have proven that
$\PP^*$ is stationary if and only if the following condition
is satisfied
$$
\forall s'\in \aS:\quad \pi^{\dag}_{s'}=
\sum\limits_{s\in \aS}\pi^{\dag}_s 
\aP^*_s (X_T=s')=\sum\limits_{s\in \aS}\pi^{\dag}_s q_{s s'}\,.
$$
This shows the theorem.
\end{proof}

Similarly as in we did in (\ref{biinf1}), when $\PP^*$ is stationary 
we can extend it to the set $\K^\ZZ$ by putting
\begin{equation}
\label{biinf2}
\PP^*(X_{l+k}=i_k, Z_{l+k}\in R'_k, k=0,..,m)=\PP^*(X_k=i_k, Z_k\in R'_k, k=0,..,m)
\end{equation}
for all $l\in \ZZ$, $m\ge 0$; $i_k\in I$, $R'_k\in \B(R)$ for $k=0,..,m$. 

\medskip

Now we state the equivalent of Theorem \ref{thmren1}
in Section \ref{secren1}. Define the probability vector
$$
\wpi^{\dag}=(\wpi^{\dag}_s: s\in \aS)
\hbox{ with } \wpi^{\dag}_s=\pi^{\dag}_s\, (\sum_{s'\in \aS}\pi^{\dag}_{s'})^{-1}.
$$
Consider the distribution  
$\aP^{\dag}_{\wpi^{\dag}}=\sum_{s\in \aS}\wpi^{\dag}_s \aP^{\dag}_s$ 
on $\B^\X_\infty$ and let $\aE^{\dag}_\wpi$ be its mean expected value.
From (\ref{impr2}) we have $\aP^{\dag}_{\wpi^{\dag}}(\T^n<\infty)=1$
for all $n\in \NN^*$, where $\T=T_{\aV\setminus \C(\X_0)}$. 
By condition (\ref{cnormxR}) we also find
$$
\aE^{\dag}_{\wpi^{\dag}}(\T)=
(\sum_{s\in \aS}\pi^{\dag}_s)^{-1}(\sum_{s\in \aS}\pi^{\dag}_s 
\aE^{\dag}_s(\T))=(\sum_{s\in \aS}\pi^{\dag}_s)^{-1}\,.
$$
Let $\aP^{* \dag}_{\wpi^{\dag}}$ be given by 
$\aP^{* \dag}_{\wpi^{\dag}}=\sum_{s\in \aS}\wpi_s 
\aP^{* \dag}_s$ on $\B^\X_\infty$ and let $\aE^{* \dag}_{\wpi^{\dag}}$ 
be its mean expected value. By previous relations, 
$$
\forall n\in \NN^*\quad
\aP^{* \dag}_{\wpi^{\dag}}(\T<\infty)=1 \hbox{ and } 
\aE^{* \dag}_{\wpi^{\dag}}(\T)^{-1}=\sum_{s\in \aS} \pi^{\dag}_s\,.
$$

The following result is proven in a similar way as we did 
for Theorem \ref{thmren1}. In fact, it is a corollary of 
Theorem \ref{statxR} because this last result allows us to construct 
the process $\X$ with distribution $\aP^{* \dag}_{\wpi^{\dag}}$ 
with independent copies between the 
sequence of hitting times of different classes. 
Since the increments of this sequence of times are 
independent identically distributed 
variables and its distribution has a finite mean 
the renewal theorem can be applied as in the proof of Theorem 
\ref{thmren1} and also the other arguments in this proof
work in an entirely analogous way. Therefore we can state:

\begin{theorem}
\label{thmren2}
Assume that $\pi^{\dag}$ satisfies (\ref{cnormxR}) and (\ref{primxR3})
and that the distribution 
of $\T$ under $ \aP^{* \dag}_{\wpi^{\dag}}$ is aperiodic.
Then,
$$
\forall B\in \B^\X_\infty:\quad
\PP^{* \dag}(B)=
\lim_{N\to \infty} \aP^{* \dag}_{\wpi^{\dag}}(B\circ \Theta_N^{-1}).
$$
Moreover, if in addition the matrix $Q^{\dag}$ is aperiodic then for 
all probability vector $\gamma=(\gamma_s: s\in \aS)$ the probability
measure $\aP^{* \dag}_\gamma=\sum_{s\in \aS} \gamma_s \aP^{* \dag}_s$ 
satisfies
$$
\forall B\in \B^\X_\infty:\quad
\PP^{* \dag}(B)=\lim_{N\to \infty}\aP_\gamma^{* \dag}(B\circ \Theta_N^{-1}).
$$
$\Box$
\end{theorem}

Let $\PP^{* \dag}$ the be law defined on $\K^\ZZ$ in (\ref{biinf2}). 
We can also show that for all $k\ge 0$, $i_k\in I$, $R'_k\in \B(R)$, $k=0,..,m$, 
we have
\begin{eqnarray*}
&{}& 
\PP^{* \dag}(X_{l+k}=i_k, Z_{l+k}\in R'_k, k=0,..,m)\\
&{}& \; =
\lim\limits_{N\to \infty} \aP_\wpi^{* \dag}(X_{l+k+N}=i_k, 
Z_{l+k+N}\in R'_k, k=0,..,m).
\end{eqnarray*}
Therefore, 
$\T^0=\sup\{\T^n: \T\le 0\}$ is finite
$\PP^{* \dag}-$a.s. and $X_{\T_0}$ has distribution $\wpi$.
Then,
\begin{eqnarray}
\nonumber
&{}& \PP^{* \dag}(X_k=i_k, Z_k\in R'_k, k=0,..,m)=\sum_{n\in \NN}
\PP^{* \dag}(\T^0=-n, X_0=i_k, k=0,..,m)\\
\label{eq22}
&{}&\; =\sum_{n\in \NN} \PP^{* \dag}(X_0=i_k, Z_k\in R'_k, k=0,..,m \, | \, \T^0=-n) 
\PP^{* \dag}(\T^0=-n).
\end{eqnarray}
Since $\PP^{* \dag}$ regenerates at each $\T^n$ with law $\wpi^\dag$, we have
$$
\PP^{* \dag}(X_0=i_k, k=0,..,m \, | \, \T^0=-n)=
\aP_{\wpi^\dag}(X_n=i_{k+n}, k=0,..,m \, | \, \T>n).
$$ 
Similarly to (\ref{eq14}) we have
$\aP^{* \dag}_{\wpi^\dag}(\T^0=-n)=
(\sum_{s\in \aS}\pi_s)\aP^*_{\wpi^\dag}(\T>n)$.
Thus, we retrieve the definition in (\ref{defr}),
$$
\PP^{* \dag}(X_0=i_k, k=0,..,m)=\sum_{s\in \aS} \sum_{n\in \NN} \pi_s 
\aP^{* \dag}_s(X_n=i_{k+n}, k=0,..,m;  \T>n).
$$
Hence, from (\ref{eq22}) we get a probabilistic 
insight to definition $\PP^{* \dag}$ and a good definition 
of $\T$ under law $\PP^{* \dag}$,
as claimed in Remark \ref{remfl2}. 

\bigskip

\begin{remark}
We have found conditions in order that $\PP^*$ or $\PP^{* \dag}$ are
stationary laws. The ergodic description of 
theses measures is part of an on-going study of the author.
\end{remark}

\section*{Acknowledgments} 

The author thanks the Center for Mathematical Modeling  (CMM) 
Basal CONICYT Program PFB 03 and INRIA-CHILE program CIRIC 
for supporting this work. He is indebted 
to Andrew Hart for fruitful discussions.

\noindent SERVET MART\'INEZ

\noindent {\it Departamento Ingenier{\'\i}a Matem\'atica and Centro
Modelamiento Matem\'atico, Universidad de Chile,
UMI 2807 CNRS, Casilla 170-3, Correo 3, Santiago, Chile.}
e-mail: smartine@dim.uchile.cl

\label{lastpage}

\end{document}